%% file: rdsa-tac-arxiv.tex
\documentclass[11pt,letterpaper,english]{article}

\usepackage[T1]{fontenc}
\usepackage[utf8]{inputenc}
\usepackage[english]{babel}
\usepackage[margin=1in]{geometry}
\usepackage[numbers]{natbib}
\usepackage[cmex10]{amsmath}
\usepackage{authblk}
\usepackage{macros}

\tikzstyle{block} = [draw, fill=white, rectangle,
    minimum height=3em, minimum width=6em]
\tikzstyle{sum} = [draw, fill=white, circle, node distance=1cm]
\tikzstyle{input} = [coordinate]
\tikzstyle{output} = [coordinate]
\tikzstyle{pinstyle} = [pin edge={to-,thin,black}]
\tikzset{roads/.style={line width=0.2cm}}

\usepackage{times}

\begin{document}

\title{Adaptive system optimization using \\random directions stochastic approximation}

\author[$\dag$]{Prashanth L.A.\thanks{prashla@isr.umd.edu}}
\author[$\sharp$]{Shalabh Bhatnagar\thanks{shalabh@csa.iisc.ernet.in}}
\author[$\$$]{Michael Fu\thanks{mfu@isr.umd.edu}}
\author[$\ddag$]{Steve Marcus\thanks{marcus@umd.edu}}
\affil[$^\dag$]{\small Institute for Systems Research, University of Maryland}
\affil[$^\sharp$]{\small Department of Computer Science and Automation,
Indian Institute of Science, Bangalore}
\affil[$^\$$]{\small Robert H. Smith School of Business \& Institute for Systems Research,
University of Maryland}
\affil[$^{\ddag}$]{\small Department of Electrical and Computer Engineering \& Institute for Systems Research,
University of Maryland}

\renewcommand\Authands{ and }

\date{}
\maketitle

\begin{abstract}
We present new algorithms for \textit{simulation optimization} using random directions stochastic approximation (RDSA). 
These include first-order (gradient) as well as second-order (Newton) schemes.
We incorporate both continuous-valued as well as discrete-valued perturbations into both our algorithms. The former are chosen to be independent and identically distributed (i.i.d.) symmetric uniformly distributed random variables (r.v.), while the latter are i.i.d., asymmetric Bernoulli r.v.s.
Our Newton algorithm, with a novel Hessian estimation scheme, requires $N$-dimensional perturbations and three loss measurements per iteration, whereas the simultaneous perturbation Newton search algorithm of \cite{spall_adaptive} requires $2N$-dimensional perturbations and four loss measurements per iteration. 
We prove the unbiasedness of both gradient and Hessian estimates and asymptotic (strong) convergence for both first-order and second-order schemes. We also provide asymptotic normality results, which in particular establish that 
the asymmetric Bernoulli variant of Newton RDSA method is better than 2SPSA of \cite{spall_adaptive}. 
Numerical experiments are used to validate the theoretical results.
% In particular, our experiments show that our Newton algorithm 2RDSA and the Newton algorithm 2SPSA in \cite{spall_adaptive} provide comparable accuracy levels in the same number of iterations, despite 2RDSA requiring only 75\% of the cost per-iteration for 2SPSA.
\end{abstract}

%\begin{IEEEkeywords}
%Stochastic optimization, stochastic approximation, random directions stochastic approximation (RDSA), simultaneous perturbation stochastic approximation (SPSA).
%\end{IEEEkeywords}

%%%%%%%%%%%%%%%%%%%%%%%%%%%%%%%%%%%%%%%%%%%%%%%%%%%%%%%%%%%%%%%%%%%%%%%%%%%%%%%%%%%%%%%%%%%%%%%%%%%%%%%%%%%%%%%%%%%%%%%%%%%%%%%%%%%%%%
%%%%%%%%%%%%%%%%%%%%%%%%%%%%%%%%%%%%%%%%%%%%%%%%%%%%%%%%%%%%%%%%%%%%%%%%%%%%%%%%%%%%%%%%%%%%%%%%%%%%%%%%%%%%%%%%%%%%%%%%%%%%%%%%%%%%%%
\section{Introduction}
\label{sec:introduction}

Problems of optimization under uncertainty arise in many areas of
engineering and science, such as signal processing, operations research,
computer networks, and manufacturing systems.
The problems themselves may involve system identification,
model fitting, optimal control, or performance evaluation based
on observed data. 
We consider the
following optimization problem in $N$ dimensions:
\begin{align}
\mbox{Find } x^* = \arg\min_{x \in \R^N} f(x). \label{eq:pb}
\end{align}
We operate in a \textit{simulation optimization} setting \cite{fu2015handbook}, i.e., we are given only noise-corrupted measurements of  $f(\cdot)$ and the goal is to devise an iterative scheme that is robust to noise.  We assume that, at any time instant $n$, the conditional expectation of the noise, say $\xi_n$, given all the randomness up to $n$ is zero (see assumption (A2) in Section \ref{sec:1rdsa-results} for the precise requirement).

A general stochastic approximation algorithm \cite{rm} 
finds the zeros of a given objective function when only noisy
estimates are available. Such a scheme can also be applied
to gradient search provided one has access to estimates of the objective function gradient.
Schemes based on sample gradients include perturbation analysis (PA) \cite{hocao}
and likelihood ratio (LR) \cite{lecuyerglynn} methods, which typically require only 
one simulation run per gradient estimate, but are not universally applicable. 

Algorithms based on gradient search perform the following iterative update:
\begin{align}
\label{eq:gd}
x_{n+1} = x_n - a_n \widehat\nabla f(x_n),
\end{align}
where $a_n$ are \textit{step-sizes} that are set in advance and satisfy standard stochastic approximation conditions (see (A2) in Section \ref{sec:1rdsa}) and $\widehat\nabla f(\cdot)$ is an estimate of the gradient $\nabla f(\cdot)$ of the objective function $f$.

The finite difference Kiefer-Wolfowitz (KW) \cite{kw} estimates require $2N$ system simulations for a gradient estimate $\widehat\nabla f$. This makes the scheme, also called finite difference
stochastic approximation (FDSA), disadvantageous for large parameter
dimensions. The random directions stochastic approximation (RDSA)
approach \cite[pp.~58-60]{kushcla} alleviates this problem by requiring two system
simulations regardless of the parameter dimension. It does this by randomly perturbing
all the $N$ parameter component directions using independent random vectors that
are uniformly distributed over the surface of the $N$-dimensional unit sphere. 
% Generating these random vectors in practice (particularly for large $N$) is however computationally difficult because of the inherent dependence between component perturbations. 
It has been observed
in \cite{chin1997comparative}, see also \cite{rubinstein}, \cite{bhat2}, \cite{bhatnagar-book}, that RDSA also works 
in the case when the component perturbations are independent Gaussian or Cauchy distributed.

RDSA with Gaussian perturbations has also been independently derived in \cite{katkul}
by approximating the gradient of the expected performance
by its convolution with a multivariate Gaussian that is then seen (via an
integration-by-parts argument) as a convolution of the objective function with a 
scaled Gaussian. This procedure requires only one simulation (regardless of
the parameter dimension) with a perturbed parameter (vector) whose component
directions are perturbed using independent standard Gaussian random variables. 
A two-simulation finite difference version that has lower bias than the
aforementioned RDSA scheme is studied in \cite{stybtang}, \cite{chin1997comparative}, \cite{bhat2}.

Among all gradient-based random perturbation approaches involving sample measurements of the objective function, the simultaneous perturbation stochastic
approximation (SPSA) of \cite{spall} has been the most popular and widely studied in applications,
largely due to its ease of implementation and observed numerical performance
when compared with other approaches.
Here each component direction of the parameter is perturbed using independent, zero mean,
symmetric Bernoulli distributed random variables. In \cite{chin1997comparative}, performance
comparisons between RDSA, SPSA, and FDSA have been studied through both analytical
and numerical comparisons of the mean square error metric, and it is observed that
SPSA outperforms both RDSA with Gaussian perturbations and FDSA.

Within the class of simulation-based search methods, there are also methods that estimate 
the Hessian in addition to the gradient. Such methods perform the following update:
\begin{align}
\label{eq:newton}
x_{n+1} = x_n - a_n (\overline H_n)^{-1}\widehat\nabla f(x_n), 
\end{align}
where $\widehat\nabla f(x_n)$ is an estimate of the gradient $\nabla f(\cdot)$ as before, while $\overline H_n$ is an $N\times N$-dimensional matrix estimating  the true Hessian $\nabla^2 f(x^*)$. Thus, \eqref{eq:newton} can be seen to be the stochastic version of the well-known Newton method for optimization. 

Stochastic Newton methods are often more accurate than simple gradient search schemes,
which are sensitive to the choice of the constant $a_0$ in the canonical step-size, $a_n= a_0/n$. The optimal (asymptotic) convergence rate is obtained only if $a_0 > 1/3 \lambda_0$, where $\lambda_0$ is the minimum eigenvalue of the Hessian of the objective function (see \cite{fabian1967stochastic}). However, this dependency is problematic, as $\lambda_0$ is unknown in a \textit{simulation optimization} setting. Hessian-based methods get rid of this dependency, while attaining the optimal rate (one can set $a_0=1$). An alternative approach to achieve the same effect is to employ Polyak-Ruppert averaging, which uses larger step-sizes and averages the iterates. However, iterate averaging is optimal only in an asymptotic sense. Finite-sample analysis (see Theorem 2.4 in \cite{fathi2013transport}) shows that the initial error (that depends on the starting point $x_0$ of the algorithm) is not forgotten sub-exponentially fast, but at the rate $1/n$ where $n$ is the number of iterations. 
Thus, 
the effect of averaging kicks in only after enough iterations have passed and the bulk of the iterates are centered around the optimum. 

In \cite{fabian}, the Hessian is estimated using $O(N^2)$ 
samples of the cost objective at each iterate, while in \cite{ruppert} the Hessian
is estimated assuming knowledge of objective function gradients. 
During the course of the last fifteen years, there has been considerable research activity
aimed at developing adaptive Newton-based random search algorithms for stochastic optimization. 
In \cite{spall_adaptive}, the first adaptive Newton algorithm using the simultaneous perturbation method was proposed. The latter algorithm involves the generation of
$2N$ independent symmetric Bernoulli distributed random variables at each update
epoch. The Hessian
estimator in this algorithm requires four parallel simulations with different perturbed
parameters at each update epoch. Two of these simulations are also used for gradient
estimation. The Hessian estimator is projected to the space of positive definite and symmetric matrices 
at each iterate for the algorithm to progress along a descent direction.
In \cite{bhat1}, three other simultaneous perturbation estimators of the Hessian
that require three, two, and one simulation(s)  have been proposed
in the context of long-run average cost objectives. The resulting algorithms incorporate
two-timescale stochastic approximation, see Chapter 6 of \cite{borkar}.
Certain three-simulation balanced simultaneous perturbation Hessian estimates have been
proposed in \cite{sbpla}. In addition, certain Hessian inversion procedures that require
lower computational effort have also been proposed, see also \cite{bhatnagar-book}.
In \cite{zhuspall}, a similar algorithm as in \cite{spall_adaptive} is considered
except that for computational simplicity, the
geometric mean of the eigenvalues (projected to the positive half line)
is used in place of the Hessian inverse
in the parameter update step. In \cite{spall-jacobian}, certain enhancements to the 
four-simulation Hessian estimates of
\cite{spall_adaptive} using some feedback and weighting mechanisms
have been proposed. 
In \cite{bhat2}, Newton-based smoothed functional algorithms based on Gaussian perturbations
have been proposed. An overview of random search approaches (both gradient and Newton-based) involving both theory and application of these techniques is available in \cite{bhatnagar-book}.

A related body of work in the machine learning community is bandit optimization \cite{hazan2015online}. Gradient estimates using the principle of first-order RDSA schemes have been used in \cite{flaxman2005online}. In \cite{nesterov2011random}, the authors explore smoothing using Gaussian perturbations for stochastic convex optimization problems and establish optimal convergence rates for this setting. In \cite{duchi2013optimal}, the authors establish optimal (non-asymptotic) convergence rates for first-order RDSA gradient estimate coupled with a mirror descent scheme. 

The principal aim of this paper is to develop an RDSA-based second-order method. 
A related objective is to achieve convergence guarantees similar to 2SPSA in \cite{spall_adaptive}, preferably at a lower per-iteration cost. 
A key ingredient in any RDSA scheme is the choice of random perturbations. We propose two schemes in this regard.
The first scheme employs i.i.d. uniform $[-\eta,\eta]$ (for some $\eta >0$) perturbations, while the second scheme employs i.i.d.  asymmetric Bernoulli perturbations. The latter takes values $-1$ and $1+\epsilon$ (for some small $\epsilon >0$) with probabilities $\left(\frac{1+\epsilon}{2+\epsilon}\right)$ and $\left(\frac{1}{2+\epsilon}\right)$, respectively.

As evident from the update rule \eqref{eq:newton}, the performance of any second-order method is affected by the choices of both the gradient estimation scheme and the Hessian estimation scheme. This motivates our study of first-order RDSA schemes with the aforementioned two perturbation schemes prior to analysing the second-order RDSA algorithms.
Table \ref{tab:algos} presents a classification of our algorithms based on their order and the perturbations employed.  
\begin{table}
\caption{A taxonomy of proposed algorithms.}
\label{tab:algos}
\centering
\begin{tabular}{c|c|c}
\textbf{Order }$\bm{\rightarrow}$ & \multirow{2}{*}{First-order} & \multirow{2}{*}{Second-order} \\ 
\textbf{Perturbations} &&\\
$\bm{\downarrow}$ &&\\\hline
\multirow{2}{*}{Uniform $\bm{[-\eta,\eta]}$} & \multirow{2}{*}{1RDSA-Unif} & \multirow{2}{*}{2RDSA-Unif} \\ 
 &&\\\hline
\multirow{2}{*}{Asymmetric Bernoulli $\bm{\{-1,1+\epsilon\}}$} & \multirow{2}{*}{1RDSA-AsymBer} & \multirow{2}{*}{2RDSA-AsymBer} \\ 
 &&\\\hline
 &&\\
\end{tabular}
\end{table}

We summarize our contributions below. 
\paragraph{\textbf{Stochastic Newton method}}
We propose an adaptive Newton-based RDSA algorithm. 
The benefits of using our procedure are two-fold. 
First,  the algorithm requires generating  only $N$ perturbation
variates at each iteration even for the Newton scheme ($N$ being the parameter dimension), 
unlike the simultaneous perturbation Newton algorithms of 
\cite{spall_adaptive}, \cite{bhat1}, \cite{spall-jacobian}, \cite{sbpla}, \cite{zhuspall}, which require generation of $2N$ perturbation
variates. Second, the number of system simulations required per iteration in our procedure is three, whereas the procedure in \cite{spall_adaptive} requires four.
%This results in significant savings in computational effort and resources.

\paragraph{\textbf{Stochastic gradient method}}
We also propose a gradient RDSA scheme 
that can be used with either uniform or asymmetric Bernoulli perturbations. Our gradient estimates require two parallel simulations
with (randomly) perturbed parameters.
% , while the Hessian estimates in the Newton procedure are also balanced and require three simulations with an additional simulation based on the nominal (running) parameter update. The overall procedure thus requires three simulations for the Newton scheme and two simulations for the gradient algorithm.

\paragraph{\textbf{Convergence and rates}}
We prove unbiasedness of our gradient and Hessian estimators and prove 
almost sure convergence of our algorithms to  local minima of the objective function. 
We also present asymptotic normality results that help us analyse our algorithms from an asymptotic mean square error (AMSE) viewpoint.
From this analysis, we observe that\\ 
\begin{inparaenum}[\bfseries(i)] 
\item On \textit{all} problem instances, the asymmetric Bernoulli variant of 2RDSA results in an AMSE that is better than 2SPSA, for the same number of iterations. 
It is computationally advantageous to employ our adaptive RDSA scheme as it requires three simulations per iteration (2SPSA requires four per iteration).\\
\item On many problem instances, the adaptive RDSA scheme with uniform perturbations that we propose results in an AMSE that is better than the Newton algorithm
(2SPSA) of \cite{spall_adaptive}. In particular, the question of whether 2SPSA or our adaptive RDSA  scheme is better depends on the values of problem-dependent (equivalently, algorithm-independent) quantities (see (A) and (B) in Section \ref{sec:amse}), and there are many problem instances where 2RDSA with uniform perturbations is indeed better than 2SPSA.\\
\item  The gradient algorithm 1RDSA with asymmetric Bernoulli perturbations exhibits an AMSE that is nearly comparable to that of SPSA, while 1RDSA variants with Gaussian \cite{kushcla,chin1997comparative} and uniform perturbations perform worse.
\end{inparaenum}

\paragraph{\textbf{Experiments}} 
Numerical results using two objective functions - one quadratic and the other fourth-order - show that \\
\begin{inparaenum}[\bfseries (i)]
\item asymmetric Bernoulli variant of 1RDSA performs on par with 1SPSA of \cite{spall}; and\\
\item our Newton algorithm 2RDSA-AsymBer provides better accuracy levels than the Newton algorithm 2SPSA in \cite{spall_adaptive} for the same number of iterations, despite 2RDSA requiring only 75\% of the cost per-iteration as compared to 2SPSA.
\end{inparaenum}

The rest of the paper is organized as follows: In Section~\ref{sec:1rdsa}, we
describe the first-order RDSA algorithm with the accompanying asymptotic theory. In Section \ref{sec:2rdsa}, we present the second-order RDSA algorithm along with proofs of convergence and asymptotic rate results.
We present the results from numerical experiments in Section~\ref{sec:experiments} and provide concluding remarks in Section~\ref{sec:conclusions}.

 %%%%%%%%%%%%%%%%%%%%%%%%%%%%%%%%%%%%%%%%%%%%%%%%%%%%%%%%
\section{First-order random directions SA (1RDSA)}
\label{sec:1rdsa}
\begin{figure}[h]
\centering
\tikzstyle{block} = [draw, fill=white, rectangle,
   minimum height=3em, minimum width=6em]
\tikzstyle{sum} = [draw, fill=white, circle, node distance=1cm]
\tikzstyle{input} = [coordinate]
\tikzstyle{output} = [coordinate]
\tikzstyle{pinstyle} = [pin edge={to-,thin,black}]
\scalebox{0.8}{\begin{tikzpicture}[auto, node distance=2cm,>=latex']
% We start by placing the blocks
\node (theta) {\large$\bm{x_n}$};
\node [sum,fill=blue!20, above right=1cm of theta, xshift=1cm] (perturb) {\large$\bm{+}$};
\node [sum,fill=red!20, below right=1cm of theta, xshift=1cm] (perturb1) {\large$\bm{-}$};
\node [above=0.5cm of perturb] (noise) {\large$\bm{\delta_n d_n}$};
\node [below=0.5cm of perturb1] (noise1) {\large$\bm{\delta_n d_n}$};    
\node [block, fill=blue!20, right=1cm of perturb,label=above:{\color{bleu2}$\bm{y_n^+}$}] (psim) {\large\bf $\bm{f(x_n+\delta_n d_n)  + \xi_n^+}$}; 
\node [block, fill=red!20, right=1cm of perturb1,label=below:{\color{bleu2}$\bm{y_n^-}$}] (sim) {\large\bf $\bm{f(x_n-\delta_n d_n)  - \xi_n^-}$}; 
\node [block,fill=green!20, below right=2cm of psim, minimum height=10em, yshift=2.7cm,text width=3cm] (update) {\large\bf{~Update $\bm{x_{n+1}}$}\\[2ex]\large\bf{~~~using \eqref{eq:1rdsa}}};
\node [right=0.7cm of update] (thetanext) {\large$\bm{x_{n+1}}$};

\draw [->] (perturb) --  (psim);
\draw [->] (perturb1) --  (sim);
\draw [->] (noise) -- (perturb);
\draw [->] (noise1) -- (perturb1);
\draw [->] (psim) -- %node {$\hat J^{\theta(t)+p_1(t)}(x_0)$}
(update.138);
\draw [->] (sim) --  %node {$\hat J^{\theta(t)+p_2(t)}(x_0)$} 
(update.217);
\draw [->] (update) -- (thetanext);
\draw [->] (theta) --   (perturb);
\draw [->] (theta) --   (perturb1);
\end{tikzpicture}}
\caption{Overall flow of 1-RDSA algorithm.}
\label{fig:1rdsa-flow}
\end{figure}
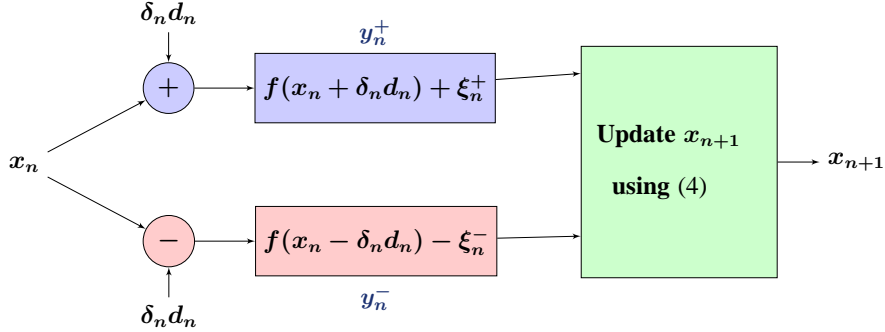

Recall that a
 first-order gradient search scheme for solving \eqref{eq:pb} has the following form: 
\begin{align}
\label{eq:1rdsa}
x_{n+1} = x_n - a_n \widehat\nabla f(x_n), 
\end{align}
where $\widehat\nabla f(x_n)$ is an estimate of $\nabla f(x_n)$.
%RDSA is a well-known random search scheme that estimates the gradient $\nabla f$ using only noisy measurements of $f$.
As illustrated in Fig. \ref{fig:1rdsa-flow}, 
the idea behind an RDSA scheme is to obtain noisy measurements of $f$ at parameter values $x_n+\delta_n d_n$ and $x_n - \delta_n d_n$. Denote these respective values by $y_n^+$ and $y_n^-$, i.e., 
$$
y_n^+ = f(x_n+\delta_n d_n) + \xi_n^+,\quad y_n^- = f(x_n-\delta_n d_n) + \xi_n^-.
$$
In the above,  the noise tuple $\{\xi_n^+, \xi_n^-, n\ge 0\}$ is a martingale difference sequence, the sequence of
the perturbation constants $\{\delta_n, n\ge 0\}$ is a positive and asymptotically vanishing sequence and
the random perturbations $d_n=(d_n^1,\ldots,d_n^N)\tr$ are such that $\{d_n^i, i=1,\ldots,N, n=1,2,\ldots\}$ are i.i.d. and independent of the noise sequence. These quantities are assumed to satisfy the conditions in (A2)-(A5) in Section \ref{sec:2rdsa-results} below. 

In the next section, we specify two different choices for $d_n$ for obtaining the gradient estimate using the noisy function measurements $y_n^+$ and $y_n^-$, respectively. The first choice uses (continuous-valued) uniform random variables, while the second is based on (discrete-valued) asymmetric Bernoulli random variates.

\subsection{Gradient estimate}

\subsubsection*{\textbf{Uniform perturbations}}
Choose $d_n^i$, $ i=1,\ldots,N$ to be i.i.d. $U[-\eta,\eta]$ for some $\eta>0$, where $U[-\eta,\eta]$ denotes the uniform distribution on the interval $[-\eta,\eta]$.
The RDSA estimate of the gradient is given by
\begin{align}
\label{eq:grad-unif}
\widehat\nabla f(x_n) = \frac3{\eta^2} d_n \left[ \dfrac{y_n^+ - y_n^-}{2\delta_n}\right].
\end{align}

\subsubsection*{\textbf{Asymmetric Bernoulli perturbations}}
Choose $d_n^i$, $i=1,\ldots,N$, i.i.d. as follows: 
\begin{equation}
\label{eq:det-proj}
 d_n^i =
  \begin{cases}
   -1 &  \text{ w.p. } \dfrac{(1+\epsilon)}{(2+\epsilon)}, \\
   1+\epsilon &  \text{ w.p. } \dfrac{1}{(2+\epsilon)},
  \end{cases}
\end{equation}
where $\epsilon>0$ is a constant that can be chosen to be arbitrarily small.
Note that $E d_n^i = 0$, $E (d_n^i)^2 = 1+\epsilon$ and $E (d_n^i)^4 = \dfrac{(1+\epsilon)(1+(1+\epsilon)^3)}{(2+\epsilon)}$.

Then, the RDSA estimate of the gradient is given by
\begin{align}
\label{eq:grad-ber}
\widehat\nabla f(x_n) = \frac1{1+\epsilon} d_n \left[ \dfrac{y_n^+ - y_n^-}{2\delta_n}\right].
\end{align}

For notational simplicity, we use $\widehat\nabla f(x_n)$ to denote the gradient estimate for both uniform and asymmetric Bernoulli distributions, where the underlying perturbations should be clear from the context.

\subsubsection*{\textbf{Motivation for the gradient estimates}}
Lemma \ref{lemma:1rdsa-bias} below establishes that the gradient estimates in \eqref{eq:grad-unif} and \eqref{eq:grad-ber} are biased by a term of order $O(\delta_n^2)$, and this bias vanishes since $\delta_n \rightarrow 0$ (see (A5) below).
 The proof uses suitable Taylor's series expansions (as in \cite{spall}) to obtain the following for both uniform and asymmetric Bernoulli perturbations:
 \begin{align*}
f(x_n \pm \delta_n d_n) = f(x_n) \pm \delta_n d_n\tr \nabla f(x_n) + \frac{\delta_n^2}{2} d_n\tr \nabla^2 f(x_n) d_n +  O(\delta_n^3).
\end{align*}
Hence, as is shown in the proof of Lemma \ref{lemma:1rdsa-bias},
$$\E\left[d_n\left.\left(\dfrac{f(x_n+\delta_n d_n) - f(x_n-\delta_n d_n)}{2\delta_n}\right)\right| \F_n \right]  =  \E\left[d_n d_n\tr \right] \nabla f(x_n)  + O(\delta_n^2),$$
where $\F_n = \sigma(x_m,m\le n)$ denotes the underlying sigma-field.
For the case of uniform perturbations, it is easy to see that $\E\left[d_n d_n\tr \right] = \frac{\eta^2}{3}$ and since we have a scaling factor of $\frac{3}{\eta^2}$ in \eqref{eq:grad-unif}, the correctness of gradient estimate follows as $\delta_n \rightarrow 0$. A similar argument holds for the case of asymmetric Bernoulli perturbations.

\begin{remark}(\textbf{Why uniform/asymmetric Bernoulli perturbations?}) 
%As described previously, we are the first to study RDSA under i.i.d., symmetric, uniformly distributed perturbations whereas 
Previous studies, see
\cite[Section 2.3.5]{kushcla} and \cite{chin1997comparative}), assumed the perturbation vector $d_n$ to be uniformly distributed over the surface of the unit sphere, whereas we consider alternative uniform/asymmetric Bernoulli perturbations for the following reasons:
\begin{description}
\item[Sample efficiency:] Let $\hat n_{\text{RDSA-Unif}}$, $\hat n_{\text{RDSA-AsymBer}}$ and $\hat n_{\text{RDSA-Gaussian}}$ denote the number of function measurements required to achieve a given accuracy using uniform, asymmetric Bernoulli and Gaussian distributed perturbations in 1RDSA, respectively. Further, let $\hat n_{\text{SPSA}}$ denote a similar number for the regular SPSA scheme with symmetric Bernoulli perturbations. Then, as discussed in detail in Section \ref{sec:amse-1rdsa}, 
we have the following ratio: 
$$ \hat n_{\text{RDSA-Unif}}:\hat n_{\text{RDSA-AsymBer}}:\hat n_{\text{RDSA-Gaussian}}:\hat n_{\text{SPSA}} = 1.8: (1+\epsilon):3:1. $$
Notice that RDSA with Gaussian perturbations requires many more measurements ($3$ times) than regular SPSA. Uniform perturbations bring down this ratio, but they are still significantly sub-optimal in comparison to SPSA. On the other hand, asymmetric Bernoulli perturbations exhibit the best ratio that can be made arbitrarily close to $1$, by choosing the distribution parameter $\epsilon$ to be a very small positive constant.

\item[Computation:]
Generating perturbations uniformly distributed over the surface of the unit sphere involves simulating $N$ Gaussian random variables, followed by normalization \cite{marsaglia1972choosing}. 
%\footnote{The classic Box-Muller method for generating Gaussians is as follows: Let $U_1, U_2 \sim U[0,1]$ (independent). Let $X_1 = \sqrt{-2 \ln U_1} \cos 2\pi U_2$ and $X_2 = \sqrt{-2 \ln U_1} \sin 2\pi U_2$. Then $X_1, X_2$ are i.i.d. standard Gaussian, i.e., $N(0,1)$.} 
In comparison, uniform/asymmetric Bernoulli perturbations  are easier to generate and 
do not involve normalization.  
\end{description}
\end{remark}

\begin{remark}(\textbf{Convex Optimization})
In \cite{duchi2013optimal}, an RDSA-based gradient estimate has been successfully employed for stochastic convex optimization, where the optimal $O(n^{-1/2})$ rate can be obtained.
 As in \cite{chin1997comparative}, the authors in \cite{duchi2013optimal} impose the condition that $\E\left[d_n d_n\tr \right] = I$, where $I$ is the identity matrix and suggest Gaussian random variables as one possibility. It is easy to see that uniform and asymmetric Bernoulli perturbations work well in the stochastic convex optimization setting as well.
\end{remark}

\subsection{Main results}
\label{sec:1rdsa-results}
% \subsection{Bias in the gradient estimate}
Recall that $\F_n = \sigma(x_m,m\le n)$ denotes the underlying sigma-field.
We  make the following assumptions\footnote{All norms are taken to be the Euclidean norm.}:
\begin{enumerate}[label=(\textbf{A\arabic*})]
\item $f:\R^N\rightarrow \R$ is three-times continuously differentiable\footnote{Here $\nabla^3 f(x) = \dfrac{\partial^3 f (x)}{\partial x\tr \partial x\tr \partial x\tr}$ denotes the third derivate of $f$ at $x$ and $\nabla^3_{i_1 i_2 i_3} f(x)$ denotes the $(i_1 i_2 i_3)$th entry of $\nabla^3 f(x)$, for $i_1, i_2, i_3=1,\ldots, N$.}  with $\left|\nabla^3_{i_1 i_2 i_3} f(x) \right| < \alpha_0 < \infty$, for $i_1, i_2, i_3=1,\ldots, N$ and for all $x\in \R^N$. 
\item $\{\xi_n^+,\xi_n^-, n=1,2,\ldots\}$ satisfy $\E\left[\left.\xi_n^+ - \xi_n^- \right| d_n, \F_n\right] = 0$.
\item For some $\alpha_1, \alpha_2, \zeta >0$ and for all $n$, 
$\E \left|\xi_n^{\pm}\right|^{2+\zeta} \le \alpha_1$, $\E \left|f(x_n\pm \delta_n d_n)\right|^{2+\zeta} \le \alpha_2$. 
\item $\{d_n^i, i=1,\ldots,N, n=1,2,\ldots\}$ are i.i.d. and independent of $\F_n$.
\item The step-sizes $a_n$ and perturbation constants $\delta_n$ are positive, for all $n$ and satisfy
$$a_n, \delta_n \rightarrow 0\text{ as } n \rightarrow \infty, 
\sum_n a_n=\infty \text{ and } \sum_n \left(\frac{a_n}{\delta_n}\right)^2 <\infty.$$
\item $\sup_n \left\| x_n \right\| < \infty$ w.p. $1$.
\end{enumerate}
The above assumptions are standard in the analysis of simultaneous perturbation methods, cf. \cite{bhatnagar-book}. In particular,
\begin{itemize}
\item  (A1) is required to ensure the underlying ODE is well-posed and also for establishing the asymptotic unbiasedness of the RDSA-based gradient estimates.  A similar assumption is required for regular SPSA as well (see Lemma 1 in \cite{spall}).
\item (A2) requires that the noise $\xi_n^+,\xi_n^-$ is a martingale difference for all $n$, while the second moment bounds in (A3) are necessary to ensure that the effect of noise can be ignored in the (asymptotic) analysis of the 1RDSA recursion \eqref{eq:1rdsa}.  
\item (A4) is crucial in establishing that the gradient estimates in \eqref{eq:grad-unif} and \eqref{eq:grad-ber} are unbiased in an asymptotic sense, because one obtains terms of the form $\E (d_n \xi_n^{\pm} \mid \F_n)$ after separating the function value $f(x_n \pm \delta_n d_n)$ and the noise $\xi_n^{\pm}$ in  \eqref{eq:grad-unif}/\eqref{eq:grad-ber}. The independence requirement in (A4) ensures that $\E (d_n (\xi_n^{+} - \xi_n^-) \mid \F_n) = \E(d_n \E ((\xi_n^{+} - \xi_n^-) \mid d_n, \F_n))=0$. See Lemma \ref{lemma:1rdsa-bias} for the proof details that utilise (A4).
\item  The step-size conditions in (A5) are standard stochastic approximation requirements, while the condition that $\sum_n \left(\frac{a_n}{\delta_n}\right)^2 <\infty$ is necessary to bound a certain martingale difference term that arises in the analysis of \eqref{eq:1rdsa}. See the proof of Theorem \ref{thm:1rdsa-strong-conv}.
\item (A6) is a stability assumption required to ensure that \eqref{eq:1rdsa} converges and is common to the analysis of stochastic approximation algorithms, which include simultaneous perturbation schemes (cf. \cite{spall,spall_adaptive,bhatnagar-book}). Note that (A6) is not straightforward to show in many scenarios. However, a standard trick to ensure boundedness is to project the iterate $x_n$ onto a compact and convex set - see  the discussion on pp. 40-41 of \cite{kushcla} and also remark E.1 of \cite{bhatnagar-book}. 
\end{itemize}

We next present three results that hold for uniform as well as asymmetric Bernoulli perturbations: First, Lemma \ref{lemma:1rdsa-bias} establishes that the bias in the gradient estimates (\ref{eq:grad-unif}) and \eqref{eq:grad-ber} is of the order $O(\delta_n^2)$. Second, Theorem \ref{thm:1rdsa-strong-conv} proves that the iterate $x_n$ governed by \eqref{eq:1rdsa} converges a.s. and finally, Theorem \ref{thm:1rdsa-asymp-normal} provides a central limit theorem-type result. 
% The proofs of Theorems \ref{thm:1rdsa-strong-conv}--\ref{thm:1rdsa-asymp-normal} are deferred to Appendix \ref{sec:appendix-ber}, as they follow in a similar fashion as that of the corresponding claims for regular SPSA. 

\begin{lemma}(\textbf{Bias in the gradient estimate})
\label{lemma:1rdsa-bias}
%Almost surely, as $\delta_n \rightarrow 0$, we have that
Under (A1)-(A6), for $\widehat\nabla f(x_n)$ defined according to either \eqref{eq:grad-unif} or \eqref{eq:grad-ber}, we have a.s. that\footnote{Here $\widehat\nabla_i f(x_n)$ and $\nabla_i f(x_n)$ denote the $i$th coordinates in the gradient estimate $\widehat\nabla f(x_n)$ and true gradient $\nabla f(x_n)$, respectively.}
\begin{align}
 \left| \E\left[\left.\widehat\nabla_i f(x_n)\right| \F_n \right] - \nabla_i f(x_n)\right| = O(\delta_n^2),  \quad \text{ for } i=1,\ldots,N.
\end{align} 
\end{lemma}

\begin{proof}
We use the proof technique of \cite{spall} (in particular, Lemma 1 there) in order to prove the main claim here.
 
Notice that
\begin{align*}
\E\left[\left.\dfrac{y_n^+ - y_n^-}{2\delta_n} \right| \F_n\right] 
= &\E\left[d_n\left.\left(\dfrac{f(x_n+\delta_n d_n) - f(x_n-\delta_n d_n)}{2\delta_n}\right) \right| \F_n\right] + \E\left[  d_n\left.\left(\dfrac{\xi_n^+ - \xi_n^-}{2\delta_n}\right) \right| \F_n\right]\\
= & \E\left[d_n\left.\left(\dfrac{f(x_n+\delta_n d_n) - f(x_n-\delta_n d_n)}{2\delta_n}\right) \right| \F_n\right].
\end{align*}
The last equality above follows from (A2) and (A4). We now analyse the term on the RHS above for both uniformly distributed perturbations and asymmetric Bernoulli perturbations.

\textbf{Case 1: Uniform perturbations}

Let $\nabla^2 f(\cdot)$ denote the Hessian of $f$.
By Taylor's series expansions, we obtain, a.s.,
\begin{align*}
f(x_n \pm \delta_n d_n) = f(x_n) \pm \delta_n d_n\tr \nabla f(x_n) + \frac{\delta_n^2}{2} d_n\tr \nabla^2 f(x_n) d_n \pm  \frac{\delta_n^3}{6} \nabla^3 f(\tilde  x_n^+)(d_n \otimes d_n \otimes d_n),
\end{align*}
where $\otimes$ denotes the Kronecker product and $\tilde x_n^+$ (resp. $\tilde x_n^-$) are on the line segment between $x_n$ and $(x_n + \delta_n d_n)$ (resp. $(x_n - \delta_n d_n)$).
Hence,
\begin{align}
&\E\left[d_n\left.\left(\dfrac{f(x_n+\delta_n d_n) - f(x_n-\delta_n d_n)}{2\delta_n}\right)\right| \F_n \right] \nonumber\\
= &\E\left[d_n d_n\tr \left.\nabla f(x_n)\right| \F_n\right]  +   \E\left[\left.\frac{\delta_n^2}{12} d_n (\nabla^3 f(\tilde  x_n^+)+\nabla^3 f(\tilde  x_n^-))(d_n \otimes d_n \otimes d_n)\right| \F_n\right]. \label{eq:l1}
\end{align}
The first term on the RHS above can be simplified as follows:
\begin{align}
\E\left[d_n d_n\tr \left.\nabla f(x_n)\right| \F_n\right] = &\E\left[d_n d_n\tr \right] \nabla f(x_n) \nonumber\\
 = & \E\left[
\begin{array}{cccc}
(d_n^1)^2 & d_n^1 d_n^2 & \cdots & d_n^1 d_n^N\\
d_n^2 d_n^1 &(d_n^2)^2 &  \cdots & d_n^2 d_n^N\\
d_n^N d_n^1 & d_n^N d_n^2 & \cdots &  (d_n^N)^2 \\
\end{array}
\right]\nabla f(x_n) =  \dfrac{\eta^2}{3} \nabla f(x_n). \label{eq:l1-a}
\end{align}
In the above, the first equality follows from (A4) and the last equality in \eqref{eq:l1-a} follows from $\E[(d_n^i)^2] = \frac{\eta^2}{3}$ and $\E[d_n^i d_n^j] = \E[d_n^i] \E[d_n^j] = 0$ for $i\ne j$.

Now, the $l$th coordinate of the second term in the RHS of \eqref{eq:l1} can be upper-bounded as follows:
\begin{align}
&\E\left[\left.\frac{\delta_n^2}{12} d_n^l (\nabla^3 f(\tilde  x_n^+)+\nabla^3 f(\tilde  x_n^-))(d_n \otimes d_n \otimes d_n)\right| \F_n\right]\nonumber\\
\le & \dfrac{\alpha_0 \delta_n^2}{6} \sum_{i_1=1}^N\sum_{i_2=1}^N\sum_{i_3=1}^N \E\left( d_n^l d_n^{i_1} d_n^{i_2} d_n^{i_3}\right)\nonumber\\
\le & \dfrac{\alpha_0 \delta_n^2\eta^4 N^3}{6}.\label{eq:l2}
\end{align}
The first inequality above follows from (A1), while the second inequality follows from the fact that $\left|d_n^l \right| \le \eta$, $l=1,\ldots,N$.
The claim follows by plugging \eqref{eq:l1-a} and \eqref{eq:l2} into \eqref{eq:l1} .\\
\textbf{Case 2: Asymmetric Bernoulli perturbations}

The proof follows in an analogous fashion as above, after noting that the scaling factor of $\frac{1}{(1+\epsilon)}$ in \eqref{eq:grad-ber} cancels out  $\E[(d_n^i)^2] = (1+\epsilon)$ and the bound in \eqref{eq:l2} gets replaced by $\left(\dfrac{\alpha_0 \delta_n^2(1+\epsilon)^4 N^3}{6}\right)$.
\end{proof}

\begin{theorem}(\textbf{Strong Convergence})
\label{thm:1rdsa-strong-conv}
Let $x^*$ be an  asymptotically stable equilibrium of the following ordinary differential equation (ODE):
$
\dot{x}_t = -\nabla f(x_t),$ with  domain of attraction $D(x^*)$, i.e., $D(x^*)= \{x_0 \mid \lim_{t\rightarrow\infty} x(t\mid x_0) = x^*\}$, where $x(t\mid x_0)$ is the solution to the ODE with initial condition $x_0$. Assume (A1)-(A6) and also that there exists a compact subset $\mathcal D$ of $D(x^*)$ such that $x_n \in \mathcal D$ infinitely often. Here $x_n$ is governed by \eqref{eq:1rdsa} with the gradient estimate $\widehat\nabla f(x_n)$ defined according to either \eqref{eq:grad-unif} or \eqref{eq:grad-ber}. Then,  we have
$$x_n \rightarrow x^* \text{ a.s. as } n\rightarrow \infty.$$ 
\end{theorem}
\begin{proof}
As in the case of regular SPSA algorithm from \cite{spall}, the proof involves verifying assumptions A2.2.1 to A2.2.3 and A2.2.4'' of \cite{kushcla} in order to invoke Theorem 2.3.1 there. The reader is referred to Appendix \ref{sec:appendix-1rdsa} for the detailed proof.  
\end{proof}

We now present an asymptotic normality result for 1RDSA, for which  we require the following variant of (A3):

\noindent
(\textbf{A3'})  The conditions of (A3) hold. In addition, $\E(\xi_n^+ - \xi_n^-)^2 \rightarrow \sigma^2$ a.s. as $n\rightarrow \infty$. 

The main result is as follows:
\begin{theorem}(\textbf{Asymptotic Normality})
\label{thm:1rdsa-asymp-normal}
Assume (A1), (A2), (A3'), (A4)-(A6). 
Let $a_n=a_0/n^\alpha$ and $\delta_n = \delta_0/n^\gamma$, where $a_0,\delta_0 >0$,  $\alpha \in (0,1]$ and $\gamma \ge 1/6$. 
Let $\beta = \alpha - 2 \gamma>0$ and $P$ be an orthogonal matrix with $P\nabla^2 f(x)P\tr = \dfrac{1}{a_0} \text{diag}\left(\lambda_1,\ldots,\lambda_N\right)$. Then, 
\begin{align}
 n^{\beta/2}(x_n - x^*) \xrightarrow{dist} \N(\mu, PMP\tr) \text{ as } n\rightarrow \infty,
\end{align}
where $\N(\mu, PMP\tr)$ denotes the multivariate Gaussian distribution with mean $\mu$ and covariance matrix $PMP\tr$. 
The mean $\mu$ is defined as follows:
$\mu=0$ if $\gamma > \alpha/6$ and $\mu = k_\mu (a_0 \delta_0^2(2 a_0 \nabla^2 f(x^*)-\beta^+I)^{-1}T)$ if $\gamma = \alpha/6$, where
\begin{equation*}
\label{eq:mu}
 k_\mu =
  \begin{cases}
   3.6  &  \text{ for uniform perturbations,}\\
    \dfrac{2(1+\epsilon)(1+(1+\epsilon)^3)}{(2+\epsilon)(1+\epsilon)^2} &  \text{ for asymmetric Bernoulli perturbations.}
  \end{cases}
\end{equation*}
In the above, $I$ is the identity matrix of size $N\times N$, $\beta^+=\beta$ if $\alpha=1$ and $0$ if $\alpha <1$ and $T = (T^1,\ldots,T^N)\tr$ with 
$$T^l = -\frac{1}{6}  \left[ \nabla^3_{lll} f(x^*) + 3 \sum\limits_{i=1,i\ne l}^{N}\nabla^3_{iil} f(x^*)\right], l=1,\ldots,N.$$   
The covariance matrix $M$ is defined as follows:
$$ M = \dfrac{a_0^2\sigma^2}{4\delta_0^2} \text{diag}((2\lambda_1-\beta^+)^{-1},\ldots,(2\lambda_N-\beta^+)^{-1}).$$
\end{theorem}
% The proof of the above theorem follows from a general result for RDSA schemes in Proposition 1 of \cite{chin1997comparative}.
\begin{proof}
Follows from Proposition 1 of \cite{chin1997comparative} after observing the following facts:

 \textit{Uniform perturbations:} $\dfrac{3}{\eta^2} \E[d_n d_n\tr] = I$ and $\dfrac{9}{\eta^4} \E[(d^i_n)^4] = 1.8$ for any $i=1,\ldots,N$.
 
 \textit{Asymmetric Bernoulli perturbations:} $\dfrac{1}{(1+\epsilon)} \E[d_n d_n\tr] = I$ and \\$\dfrac{1}{(1+\epsilon)^2} \E[(d^i_n)^4] = \dfrac{(1+\epsilon)(1+(1+\epsilon)^3)}{(2+\epsilon)(1+\epsilon)^2}$ for any $i=1,\ldots,N$.
\end{proof}

\subsection{(Asymptotic) convergence rates}
\label{sec:amse-1rdsa}
The result in Theorem \ref{thm:1rdsa-asymp-normal} shows that $n^{\beta/2}(x_n - x^*)$ is asymptotically Gaussian for 1RDSA under both perturbation choices. 
The asymptotic mean square error of $n^{\beta/2}(x_n - x^*)$, denoted by $\mathcal{AMSE}_{1\mathcal{RDSA}}(a,c)$, is given by
$$\mathcal{AMSE}_{1\mathcal{RDSA}}(a_0,\delta_0) = \mu\tr\mu + \text{trace}(PMP\tr),$$
where $a_0$ is the step-size constant, $\delta_0$ is the constant in the perturbation sequence $\delta_k$ and $\mu, P$ and $M$ are as defined in Theorem \ref{thm:1rdsa-asymp-normal}. 
Under certain assumptions (cf. \cite{gerencser1999convergence}), it can be shown that $\mathcal{AMSE}_{1\mathcal{RDSA}}(a,c)$ coincides with $n^\beta\E\l x_n-x^*\r^2$. From Theorem \ref{thm:1rdsa-asymp-normal}, it is easy to deduce from the conditions on step-size exponent  $\alpha$ and perturbation constant  exponent $\gamma$ that the range of $\beta$ is $0$ to $2/3$. Following the discussion in Section III-A of \cite{chin1997comparative}, a common value of $\beta=2/3$ is optimal for all first-order algorithms, with $\alpha=1$ and $\gamma = 1/6$. 

With step-size $a_n = a_0/n$, setting $a_0$ optimally requires knowledge of the minimum eigenvalue $\lambda_0$ of the Hessian $\nabla^2 f(x^*)$, i.e., $a_0 > \beta/2\lambda_0$.
Under this choice, we obtain
\begin{align*}
 \mathcal{AMSE}_{1\mathcal{RDSA-}\textit{Unif}}(a_0,\delta_0) =& \left(3.6 \delta_0^2 a_0 \l (2a_0\nabla^2 f(x^*) - \beta)^{-1} T \r \right)^2 
\\&+ \delta_0^{-2} \text{trace}\left( (2a_0 \nabla^2 f(x^*) - \beta)^{-1} S\right),
\end{align*}
where $T$ is as defined in Theorem \ref{thm:1rdsa-asymp-normal} and $S= \dfrac{\sigma^2}{4} I$.

\begin{remark}\textbf{\textit{(On step-size dependency)}}
Since  $\lambda_0$ is unknown, obtaining the above rate is problematic and one can get rid of the dependency of $a_0$ on $\lambda_0$ either by averaging of iterates or employing an adaptive (second-order) scheme. The former would employ step-size $a_n = a_0/n^\alpha$, with $\alpha \in (1/2,1)$ and couple this choice with averaging of iterates as $\bar x_n = 1/n \sum_{m=1}^{n} x_m$. The latter adaptive scheme would correspond to 2RDSA, which performs a Newton step to update $x_n$ in \eqref{eq:2rdsa}. Section \ref{sec:2rdsa} presents 2RDSA along with an AMSE analysis that compares to 2SPSA.
\end{remark}

\subsubsection*{\textbf{Comparing AMSE of 1RDSA-Unif to that of 1SPSA}}
Taking the ratio of AMSE of 1RDSA with uniform perturbations to that of 1SPSA with symmetric Bernoulli $\pm 1$-valued perturbations, we obtain:
\begin{align*}
 &\dfrac{\mathcal{AMSE}_{1\mathcal{RDSA-}\textit{Unif}}(a_0,\delta_0)}{\mathcal{AMSE}_{1\mathcal{SPSA}}(a_0,\delta_0)} \\
 =& \dfrac{\left(2 \delta_0^2 a_0 \l (2a_0 \nabla^2 f(x^*) - \beta)^{-1} T \r 1.8 \right)^2 + a_0\delta_0^{-2} \text{trace}\left( (2a_0 \nabla^2 f(x^*) - \beta)^{-1} S\right)}{\left(2 \delta_0^2 a_0 \l (2a_0 \nabla^2 f(x^*) - \beta)^{-1} T \r\right)^2 + a_0\delta_0^{-2} \text{trace}\left( (2a_0 \nabla^2 f(x^*) - \beta)^{-1} S\right)}\\
=& 1 + \dfrac{2.24}{1+\left(a_0\delta_0^{-2} \text{trace}\left( (2a_0 \nabla^2 f(x^*) - \beta)^{-1} S\right)\right)/\left(2 \delta_0^2 a_0 \l (2a_0 \nabla^2 f(x^*) - \beta)^{-1} T \r \right)^2}.
\end{align*}
From the above, we observe that 1SPSA has a better AMSE in comparison to 1RDSA, but it is not clear if the difference is `large'. This is because the ratio in the denominator above  depends on the objective function (via $\nabla^2 f(x^*)$ and $T$) and a high ratio value would make the difference between 1RDSA and 1SPSA negligible. Contrast this with the $1.8:1$ ratio obtained if one knows the underlying objective function (see Remark \ref{remark:1rdsa-gaussian} below).

\subsubsection*{\textbf{Comparing AMSE of 1RDSA-AsymBer  to that of 1SPSA}}
Observe that $\dfrac{1}{(1+\epsilon)} \E[d_n d_n\tr] = I$ and so, $\dfrac{1}{(1+\epsilon)^2} \E[(d^i_n)^4] = \dfrac{\beta}{(1+\epsilon)^2}$ for any $i=1,\ldots,N$. If we used $U[-1,1]$ r.v. for perturbations, then $9\E[(d^i_n)^4] = \frac{9}{5}=1.8$ and this value causes AMSE of 1RDSA to be much higher than that of 1SPSA.

On the other hand, choosing $\epsilon = 0.01$ for 1RDSA with asymmetric Bernoulli perturbations, we obtain $\dfrac{1}{(1+\epsilon)^2} \E[(d^i_n)^4] = \dfrac{\beta}{(1+\epsilon)^2} = 1.000099$. Plugging this value into the AMSE calculation, we obtain:
\begin{align*}
 &\dfrac{\mathcal{AMSE}_{1\mathcal{RDSA-}\textit{AsymBer}}(a_0,\delta_0)}{\mathcal{AMSE}_{1\mathcal{SPSA}}(a_0,\delta_0)} \\
 =& \dfrac{\left(2 \delta_0^2 a_0 \l (2a_0 \nabla^2 f(x^*) - \beta)^{-1} T \r 1.000099 \right)^2 + a_0\delta_0^{-2} \text{trace}\left( (2a_0 \nabla^2 f(x^*) - \beta)^{-1} S\right)}{\left(2 \delta_0^2 a_0 \l (2a_0 \nabla^2 f(x^*) - \beta)^{-1} T \r\right)^2 + a_0\delta_0^{-2} \text{trace}\left( (2a_0 \nabla^2 f(x^*) - \beta)^{-1} S\right)}
%=& 1 + \dfrac{2.24}{1+\left(\delta_0^{-2} \text{trace}\left( (2a_0 \nabla^2 f(x^*) - \beta)^{-1} S\right)\right)/\left(2 \delta_0^2 a_0 \l (2a_0 \nabla^2 f(x^*) - \beta)^{-1} T \r \right)^2}
\end{align*}
From the above, we observe that 1RDSA has an AMSE that is almost comparable to that of 1SPSA. One could choose a small $\epsilon$ to get this ratio arbitrarily close to $1$.

\begin{remark}(\textbf{1RDSA with Gaussian perturbations})
\label{remark:1rdsa-gaussian}
In \cite{chin1997comparative}, the author simplifies the AMSE for 1RDSA by solving  $\mathcal{AMSE}_{1\mathcal{RDSA}}(a_0,\delta_0)$ for $\delta_0$ after setting $a_0$ optimally using $\lambda_0$. Using $N(0,1)$ for $d_n$ and comparing the resulting AMSE of 1RDSA to that of first-order SPSA with symmetric Bernoulli distributed perturbations, they report a ratio of $3:1$ for the number of measurements to achieve a given accuracy. Here $3$ is a result of the fact that for $N(0,1)$ distributed $d_n$, $\E d_n^4=3$, while $1$ for SPSA comes from a bound on the second and inverse second moments, both of which are $1$ for the Bernoulli case. 
Using $U[-\eta,\eta]$ distributed $d_n$ in 1RDSA would bring down this ratio to $1.8:1$. However, this result comes with a huge caveat - that $a_0$ and $\delta_0$ are set optimally. Setting these quantities requires knowledge of the objective, specifically, $\nabla^2 f(x^*)$ and the vector $T$.
\end{remark}

%%%%%%%%%%%%%%%%%%%%%%%%%%%%%%%%%%%%%%%%%%%%%%%%%%%%%%%%%%%%%%%%%%%%%%%%%%%%%%%%%%%%%%
\section{Second-order random directions SA (2RDSA)}
\label{sec:2rdsa}
As in the case of the first-order scheme, we present two variants of the second-order simulation optimization method: one using uniform perturbations and the other using asymmetric Bernoulli perturbations. 
Recall that a second-order adaptive search algorithm has the following form \cite{spall_adaptive}:
\begin{align}
\label{eq:2rdsa}
x_{n+1} = x_n - a_n \Upsilon(\overline H_n)^{-1}\widehat\nabla f(x_n), \\
\overline H_n = \frac{n}{n+1} \overline H_{n-1} + \frac{1}{n+1} \widehat H_n.\label{eq:2rdsa-H}
\end{align}
In the above, 
\begin{itemize}
 \item $\widehat\nabla f(x_n)$ is the estimate of $\nabla f(x_n)$ and this corresponds to \eqref{eq:grad-unif} for the uniform variant and \eqref{eq:grad-ber} for the asymmetric Bernoulli variant.
 \item $\widehat H_n$ is an estimate of the true Hessian ${\nabla}^2 f(\cdot)$, with $\widehat H_0 = I$. 
 \item $\overline H_n$ is a smoothed version of $\widehat H_n$, which is crucial to ensure convergence. 
 \item $\Upsilon$ is an operator that projects a matrix onto the set of positive definite matrices. Update \eqref{eq:2rdsa-H} does not necessarily ensure that $\overline H_n$ is invertible and without $\Upsilon$, the parameter update \eqref{eq:2rdsa} may not move along a descent direction - see conditions (C7) and (C12) in Section \ref{sec:2rdsa-results} below for the precise requirements on the matrix projection operator.
\end{itemize}

The basic algorithm in \eqref{eq:2rdsa}--\eqref{eq:2rdsa-H} is similar to the adaptive scheme analyzed by \cite{spall_adaptive}. However, we use RDSA for the gradient and Hessian estimates, while \cite{spall_adaptive} employs SPSA.

\begin{remark}\textbf{\textit{(Matrix projection)}}\label{remark:upsilon}
A simple way to define $\Upsilon(\overline H_n)$ is to first perform an eigen-decomposition of $\overline H_n$, followed by projecting all the eigenvalues onto the positive real line by adding a positive scalar $\delta_n$ - see \cite{gill1981practical}, \cite{spall_adaptive} for a similar operator. This choice for $\Upsilon$ satisfies the assumptions (C7) and (C12), which are required to ensure asymptotic unbiasedness of the Hessian scheme presented in the next section. 
Note that the scalar $\delta_n$ used for $\Upsilon$ is also used as a perturbation constant for function evaluations (see \eqref{eq:xi} below). 
\end{remark}

\begin{figure}[h]
\centering
\tikzstyle{block} = [draw, fill=white, rectangle,
   minimum height=3em, minimum width=6em]
\tikzstyle{sum} = [draw, fill=white, circle, node distance=1cm]
\tikzstyle{input} = [coordinate]
\tikzstyle{output} = [coordinate]
\tikzstyle{pinstyle} = [pin edge={to-,thin,black}]
\scalebox{0.75}{\begin{tikzpicture}[auto, node distance=2cm,>=latex']
% We start by placing the blocks
\node (theta) {\large$\bm{x_n}$};
\node [sum,fill=blue!20, above right=2.5cm of theta, xshift=1cm] (perturb) {\large$\bm{+}$};
\node [sum,fill=red!20, below right=2.5cm of theta, xshift=1cm] (perturb1) {\large$\bm{-}$};
\node [above=0.5cm of perturb] (noise) {\large$\bm{\delta_n d_n}$};
\node [below=0.5cm of perturb1] (noise1) {\large$\bm{\delta_n d_n}$};    
\node [block, fill=blue!20, right=1cm of perturb,label=above:{\color{bleu2}$\bm{y_n^+}$}] (psim) {\large\bf $\bm{f(x_n+\delta_n d_n)  + \xi_n^+}$}; 
\node [block, fill=red!20, right=1cm of perturb1,label=below:{\color{bleu2}$\bm{y_n^-}$}] (sim) {\large\bf $\bm{f(x_n-\delta_n d_n)  - \xi_n^-}$}; 
\node [block, fill=brown!20, right=5.5cm of theta,label=above:{\color{bleu2}$\bm{y_n}$}] (unpsim) {\large\bf $\bm{f(x_n) + \xi_n}$}; 
\node [block,fill=green!20, below right=4cm of psim, minimum height=12.5em, yshift=3.5cm,text width=3cm] (update) {\large\bf{~~~Update }\\[2ex]\large\bf{~~using \eqref{eq:2rdsa}}};
\node [right=0.7cm of update] (thetanext) {\large$\bm{x_{n+1}}$};

\draw [->] (perturb) --  (psim);
\draw [->] (perturb1) --  (sim);
\draw [->] (noise) -- (perturb);
\draw [->] (noise1) -- (perturb1);
\draw [->] (psim) -- %node {$\hat J^{\theta(t)+p_1(t)}(x_0)$}
(update.135);
\draw [->] (sim) --  %node {$\hat J^{\theta(t)+p_2(t)}(x_0)$} 
(update.220);
\draw [->] (update) -- (thetanext);
\draw [->] (theta) --   (unpsim);
\draw [->] (unpsim) --   (update);
\draw [->] (theta) --   (perturb);
\draw [->] (theta) --   (perturb1);
\end{tikzpicture}}
\caption{Overall flow of 2-RDSA algorithm.}
\label{fig:2rdsa-flow}
\end{figure}
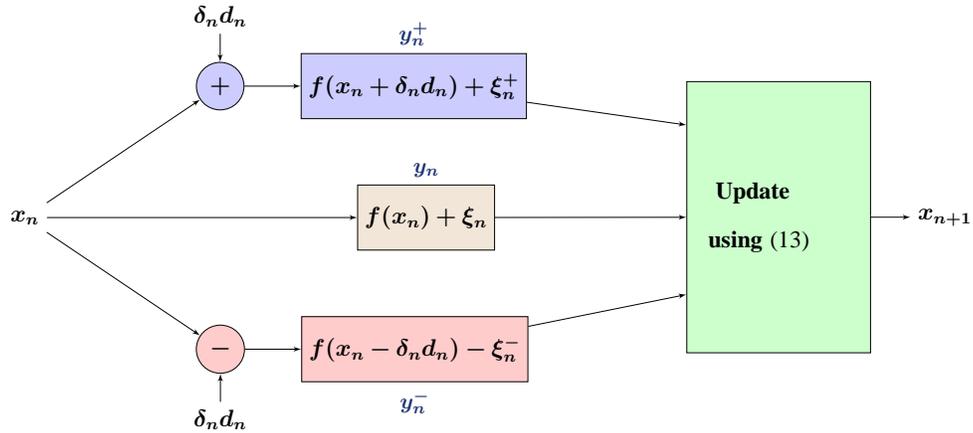

\subsection{Hessian estimate}
As illustrated in Fig. \ref{fig:2rdsa-flow}, we use three measurements per iteration in \eqref{eq:2rdsa} to estimate both the gradient and the Hessian of the objective $f$. These measurements correspond to parameter values $x_n$, $x_n+\delta_n d_n$ and $x_n - \delta_n d_n$. Let us denote these values by  $y_n$, $y_n^+$ and $y_n^-$ respectively, i.e., 
\begin{align}
\label{eq:xi}
y_n = f(x_n) + \xi_n, \quad y_n^+ = f(x_n+\delta_n d_n) + \xi_n^+,\quad y_n^- = f(x_n-\delta_n d_n) + \xi_n^-,
\end{align}
where the noise terms $\xi_n, \xi_n^+, \xi_n^-$ satisfy $\E\left[\left.\xi_n^+ + \xi_n^- - 2\xi_n \right| \F_n\right] = 0$. We next present two constructions for the perturbations $d_n$ - one based on i.i.d. uniform r.v.s and the other using asymmetric Bernoulli r.v.s. Unlike the construction in \cite{spall_adaptive}, which entails generating $2N$ Bernoulli r.v.s  in each iteration, our construction requires $N$ r.v.s that follow either a uniform or an asymmetric Bernoulli distribution.

\subsubsection*{\textbf{Uniform perturbations}}
Using the three measurements and the random directions obtained from $d_n$, we form the Hessian estimate $\widehat H_n$ as follows:
\begin{align}
\label{eq:2rdsa-estimate-unif}
&\widehat H_n = \dfrac{9}{2\eta^4} M_n \left(\dfrac{y_n^+ + y_n^- - 2 y_n}{\delta_n^2}\right), \\
&\hspace{-5em}\text{ where } M_n =
\left[
\begin{array}{cccc}
\frac{5}{2}\left((d_n^1)^2-\frac{\eta^2}{3}\right) & d_n^1 d_n^2 & \cdots & d_n^1 d_n^N\\
d_n^2 d_n^1 &\frac{5}{2}\left((d_n^2)^2-\frac{\eta^2}{3}\right) &  \cdots & d_n^2 d_n^N\\
d_n^N d_n^1 & d_n^N d_n^2 & \cdots &  \frac{5}{2}\left((d_n^N)^2-\frac{\eta^2}{3}\right) \\
\end{array}
\right].\nonumber
\end{align}
%Henceforth, we shall refer to algorithm \eqref{eq:2rdsa}--\eqref{eq:2rdsa-H} with Hessian estimate \eqref{eq:2rdsa-estimate} as 2RDSA.
Lemma \ref{lemma:2rdsa-bias} establishes that the above estimator is of order $O(\delta_n^2)$ away from the true Hessian. The first step of the proof is to  use a suitable Taylor's series expansion of $f$ to obtain
\begin{align}
&\dfrac{f(x_n+\delta_n d_n) + f(x_n-\delta_n d_n)  - 2 f(x_n)}{\delta_n^2} \nonumber\\
=&  \sum\limits_{i=1}^N (d_n^i)^2 \nabla^2_{ii} f(x_n) + 2\sum\limits_{i=1}^{N-1}\sum\limits_{j=i+1}^N d_n^i d_n^j \nabla^2_{ij} f(x_n) + O(\delta_n^2). \nonumber
\end{align}
Taking conditional expectations on both sides above, it can be seen that the RHS does not simplify to be true Hessian and includes bias terms. By multiplying the term $\dfrac{9}{2\eta^4} M_n$ with the RHS above, we obtain an (asymptotically) unbiased Hessian estimate - see the passage starting from \eqref{eq:h1} in the proof of Lemma \ref{lemma:2rdsa-bias} below for details.

\subsubsection*{\textbf{Asymmetric Bernoulli perturbations}}
Using the three measurements and the random directions obtained from $d_n$, we form the Hessian estimate $\widehat H_n$ as follows:
\begin{align}
\label{eq:2rdsa-estimate-ber}
&\widehat H_n = M_n \left(\dfrac{y_n^+ + y_n^- - 2 y_n}{\delta_n^2}\right),\\
&\hspace{-3em}\text{ where } M_n =
\left[
\begin{array}{cccc}
\frac{1}{\kappa}\left((d_n^1)^2-(1+\epsilon)\right) & \frac{1}{2(1+\epsilon)^2}d_n^1 d_n^2 & \cdots & \frac{1}{2(1+\epsilon)^2}d_n^1 d_n^N\\
\frac{1}{2(1+\epsilon)^2}d_n^2 d_n^1 &\frac{1}{\kappa}\left((d_n^2)^2-(1+\epsilon)\right) &  \cdots & \frac{1}{2(1+\epsilon)^2}d_n^2 d_n^N\\
\frac{1}{2(1+\epsilon)^2}d_n^N d_n^1 & \frac{1}{2(1+\epsilon)^2}d_n^N d_n^2 & \cdots &  \frac{1}{\kappa}\left((d_n^N)^2-(1+\epsilon)\right) \\
\end{array}
\right],\nonumber
\end{align}
where $\kappa = E (d_n^i)^4 \left(1- \dfrac{(1+\epsilon)^2}{\beta}\right)$, with $E (d_n^i)^4= \dfrac{(1+\epsilon)(1+(1+\epsilon)^3)}{(2+\epsilon)}$ denoting the fourth moment $E (d_n^i)^4$, $i=1,\ldots,N$.

\begin{remark}\textbf{(\textit{Need for asymmetry})}
The Hessian estimate that we construct is such that it disallows using symmetric Bernoulli r.v.s for perturbations. In particular, in establishing the unbiasedness of the Hessian estimate, the proof requires that the second and fourth moments of the perturbation r.v.s be different. Asymmetric Bernoulli r.v.s (with only a slight asymmetry) meet this condition and can be used in deriving a gradient estimate as well. 
%The latter scheme for the gradient is almost comparable to that of SPSA and hence, the overall 2RDSA procedure with asymmetric Bernoulli perturbations is superior to 2SPSA of \cite{spall_adaptive}. 
\end{remark}

\subsection{Main results}
\label{sec:2rdsa-results}
Recall that $\F_n = \sigma(x_m,m\le n)$ denotes the underlying sigma-field. 
We make the following assumptions that are similar to those in \cite{spall_adaptive}:
\begin{enumerate}[label=(\textbf{C\arabic*})]
\item  The function
$f$ is four-times differentiable\footnote{Here $\nabla^4 f(x) = \dfrac{\partial^4 f (x)}{\partial x\tr \partial x\tr \partial x\tr \partial x\tr}$ denotes the fourth derivate of $f$ at $x$ and $\nabla^4_{i_1 i_2 i_3 i_4} f(x)$ denotes the $(i_1 i_2 i_3 i_4)$th entry of $\nabla^4 f(x)$, for $i_1, i_2, i_3,i_4=1,\ldots, N$.} with $\left|\nabla^4_{i_1 i_2 i_3 i_4} f(x) \right| < \infty$, for $i_1, i_2, i_3,i_4=1,\ldots, N$ and for all $x\in \R^N$. 

%\item  For some $\rho>0$  and almost all $x_n$, the function $f$ is four-times differentiable with a uniformly (in $n$) bounded fourth derivative for all $x$ such that $\left\| x_n - x\right\| \le \rho$. 

\item For each $n$ and all $x$, there exists a $\rho>0$ not dependent on $n$ and $x$, such that $(x-x^*)\tr \bar f_n(x) \ge \rho \left\| x_n - x\right\|$, where $\bar f_n(x) = \Upsilon(\overline H_n)^{-1} \nabla f(x)$.

\item $\{\xi_n, \xi_n^+,\xi_n^-, n=1,2,\ldots\}$ satisfy $\E\left[\left. \xi_n^+ + \xi_n^- - 2 \xi_n \right| \F_n\right] = 0$, for all $n$. 

\item  Same as (A4). %$\{d_n^i, i=1,\ldots,N, n=1,2,\ldots\}$ are i.i.d. and independent of $\F_n$.

\item Same as (A5).

\item For each $i=1,\ldots,N$ and any $\rho>0$, 
$$P(\{ \bar f_{ni} (x_n) \ge 0 \text{ i.o}\} \cap \{ \bar f_{ni} (x_n) < 0 \text{ i.o}\} \mid \{ |x_{ni} - x^*_i| \ge \rho\quad \forall n\}) =0.$$

\item The operator $\Upsilon$ satisfies $\delta_n^2 \Upsilon(H_n)^{-1} \rightarrow 0$ a.s. and  $E(\left\| \Upsilon(H_n)^{-1}\right\|^{2+\zeta}) \le \rho$ for some $\zeta, \rho>0$.

\item For any $\tau >0$ and nonempty $S \subseteq \{1,\ldots,N\}$, there exists a $\rho'(\tau,S)>\tau$ such that 
$$ \limsup_{n\rightarrow \infty} \left| \dfrac{\sum_{i \notin S} (x-x^*)_i \bar f_{ni}(x)}{\sum_{i \in S} (x-x^*)_i \bar f_{ni}(x)}               \right| < 1 \text{ a.s.}$$
for all $|(x-x^*)_i| < \tau$ when $i \notin S$ and $|(x-x^*)_i| \ge \rho'(\tau,S)$ when $i\in S$.
\item For some $\alpha_0, \alpha_1>0$ and for all $n$, $\E {\xi_n}^{2} \le \alpha_0$, $\E {\xi_n^{\pm}}^{2} \le \alpha_0$, $\E f(x_n)^{2} \le \alpha_1$ and $\E f(x_n\pm \delta_n d_n)^{2} \le \alpha_1$. 
\item  $\sum_n \frac{1}{(n+1)^{2}\delta_n^{2}} < \infty$.
\end{enumerate}
For a detailed interpretation of the above conditions, the reader is referred to Section III and Appendix B of \cite{spall_adaptive}. In particular, 
(C1) holds if the objective $f$ is twice continuously differentiable with a bounded second derivative and (C2) ensures the objective $f$ has enough curvature.
(C3)-(C5) are standard requirements on noise and step-sizes and can be motivated in a similar manner as in the case of 1RDSA (see Section \ref{sec:1rdsa-results}).  (C6) and (C8) are not necessary if the iterates are bounded, i.e., $\sup_n \l x_n \r < \infty$ a.s. (C7) can be ensured by having $\Upsilon$ defined as mentioned earlier, i.e., $\Upsilon(A)$ performs an eigen-decomposition of $A$ followed by projecting the eigenvalues to the positive side by adding a large enough scalar. Finally, (C9) and (C10) are necessary to ensure convergence of the Hessian recursion, in particular, to invoke a martingale convergence result (see Theorem \ref{thm:2rdsa-H} and its proof below).

\begin{lemma}(\textbf{Bias in Hessian estimate})
\label{lemma:2rdsa-bias}
Under (C1)-(C10), with $\widehat H_n$ defined according to either \eqref{eq:2rdsa-estimate-unif} or \eqref{eq:2rdsa-estimate-ber}, we have a.s. that\footnote{Here $\widehat H_n(i,j)$ and $\nabla^2_{ij}f(\cdot)$ denote the $(i,j)$th entry in the Hessian estimate $\widehat H_n$ and the true Hessian $\nabla^2 f(\cdot)$, respectively.}, for $i,j = 1,\ldots,N$,
\begin{align}
\left|\E\left[
\left. \widehat H_n(i,j) \right| \F_n \right] - \nabla^2_{ij} f(x_n)\right| = O(\delta_n^2).
\end{align} 
\end{lemma}
From the above lemma, it is evident that the bias in the Hessian estimate above is of the same order as 2SPSA of \cite{spall_adaptive}. 
\begin{proof}\textbf{\textit{(Lemma \ref{lemma:2rdsa-bias})}}\ \\
\textbf{Case 1: Uniform perturbations}

By a Taylor's series expansion, we obtain
\begin{align*}
f(x_n \pm \delta_n d_n) = &f(x_n) \pm \delta_n d_n\tr \nabla f(x_n) + \frac{\delta_n^2}{2} d_n\tr \nabla^2 f(x_n) d_n \pm \frac{\delta_n^3}{6} \nabla^3 f(x_n) (d_n \otimes d_n \otimes d_n)\\
&  +  \frac{\delta_n^4}{24} \nabla^4 f(\tilde  x_n^+)(d_n \otimes d_n \otimes d_n \otimes d_n).
\end{align*}
The fourth-order term in each of the expansions above can be shown to be of order $O(\delta_n^4)$ using (C1) and arguments similar to that in Lemma \ref{lemma:1rdsa-bias} (see \eqref{eq:l2} there). Hence,
\begin{align*}
\dfrac{f(x_n+\delta_n d_n) + f(x_n-\delta_n d_n) - 2 f(x_n)}{\delta_n^2} =& d_n\tr \nabla^2 f(x_n) d_n +  O(\delta_n^2)\\
= & \sum\limits_{i=1}^N\sum\limits_{j=1}^N d_n^i d_n^j \nabla^2_{ij} f(x_n) + O(\delta_n^2)\\
=& \sum\limits_{i=1}^N (d_n^i)^2 \nabla^2_{ii} f(x_n) + 2\sum\limits_{i=1}^{N-1}\sum\limits_{j=i+1}^N d_n^i d_n^j \nabla^2_{ij} f(x_n) + O(\delta_n^2).
\end{align*}
Now, taking the conditional expectation of the Hessian estimate $\widehat{H_n}$ and observing that $\E[\xi_n^+ + \xi_n^- - 2\xi_n \mid \F_n] = 0$ by (C3), we obtain the following:
\begin{align}
\E[\widehat H_n \mid \F_n] =   \E\left[\left. M_n \left(\sum\limits_{i=1}^{N-1} (d_n^i)^2 \nabla^2_{ii} f(x_n) + 2\sum\limits_{i=1}^N\sum\limits_{j=i+1}^N d_n^i d_n^j \nabla^2_{ij} f(x_n) + O(\delta_n^2)\right)\right| \F_n\right]. \label{eq:h1}
\end{align}

Note that the $O(\delta_n^2)$ term inside the conditional expectation above remains $O(\delta_n^2)$ even after the multiplication with $M_n$.
We analyse the diagonal and off-diagonal terms in the multiplication of the matrix $M_n$ with the scalar above, ignoring the $O(\delta_n^2)$ term. 

\subsection*{Diagonal terms in \eqref{eq:h1}:}

Consider the $lth$ diagonal term inside the conditional expectation in \eqref{eq:h1}:
\begin{align}
& \dfrac{45}{4\eta^4}  \left((d_n^l)^2-\frac{\eta^2}{3}\right) \left(\sum\limits_{i=1}^N (d_n^i)^2 \nabla^2_{ii} f(x_n) + 2\sum\limits_{i=1}^{N-1}\sum\limits_{j=i+1}^N d_n^i d_n^j \nabla^2_{ij} f(x_n)\right)\nonumber\\
=& \dfrac{45}{4\eta^4} (d_n^l)^2 \sum\limits_{i=1}^N (d_n^i)^2 \nabla^2_{ii} f(x_n) + \dfrac{45}{2\eta^4} (d_n^l)^2\sum\limits_{i=1}^{N-1}\sum\limits_{j=i+1}^N d_n^i d_n^j \nabla^2_{ij} f(x_n)\nonumber\\
& - \dfrac{15}{4\eta^2} \sum\limits_{i=1}^N (d_n^i)^2 \nabla^2_{ii} f(x_n) - \dfrac{15}{2\eta^2} \sum\limits_{i=1}^{N-1}\sum\limits_{j=i+1}^N d_n^i d_n^j \nabla^2_{ij} f(x_n).\label{eq:h2}
\end{align}
From the distributions of $d_n^i,d_n^j$ and the fact that $d_n^i$ is independent of $d_n^j$ for $i<j$, it is easy to see that $\E\left(\left. (d_n^l)^2\sum\limits_{i=1}^{N-1}\sum\limits_{j=i+1}^N d_n^i d_n^j \nabla^2_{ij} f(x_n) \right| \F_n\right) = 0$ and $\E\left(\left.\sum\limits_{i=1}^{N-1}\sum\limits_{j=i+1}^N d_n^i d_n^j \nabla^2_{ij} f(x_n) \right| \F_n\right) = 0$. Thus, the conditional expectations of the second and fourth terms on the RHS of \eqref{eq:h2} are both zero. 

The first term on the RHS of \eqref{eq:h2} with the conditional expectation can be simplified as follows:
\begin{align*}
\dfrac{45}{4\eta^4} \E\left(\left. (d_n^l)^2 \sum\limits_{i=1}^N (d_n^i)^2 \nabla^2_{ii} f(x_n) \right| \F_n\right)
= & \dfrac{45}{4\eta^4} \E\left((d_n^l)^4 \nabla^2_{ll} f(x_n) + \sum\limits_{i=1,i\ne l}^N (d_n^l)^2(d_n^i)^2 \nabla^2_{ii} f(x_n) \right)\\
= & \dfrac{45}{4\eta^4} \left( \dfrac{\eta^4}{5} \nabla^2_{ll} f(x_n) + \dfrac{\eta^4}{9}\sum\limits_{i=1,i\ne l}^N  \nabla^2_{ii} f(x_n) \right), \text{ a.s.}
\end{align*} 
For the second equality above, we have used the fact that $\E[(d_n^l)^4] =  \frac{\eta^4}{5}$ and $\E[(d_n^l)^2 (d_n^i)^2] = \E[(d_n^l)^2] \E[(d_n^i)^2] = \dfrac{\eta^4}{9}$, $\forall l \ne i$.

The third term in \eqref{eq:h2} with the conditional expectation and without the negative sign can be simplified as follows: 
\begin{align*}
\dfrac{15}{4\eta^2} \E\left(\left. \sum\limits_{i=1}^N (d_n^i)^2 \nabla^2_{ii} f(x_n) \right| \F_n\right)
= & \dfrac{15}{4\eta^2} \sum\limits_{i=1}^N \E \left[(d_n^i)^2\right] \nabla^2_{ii} f(x_n) \\
= & \dfrac{5}{4} \sum\limits_{i=1}^N \nabla^2_{ii} f(x_n), \text{ a.s.}
\end{align*} 
Combining the above followed by some algebra, we obtain 
\begin{align*}
\dfrac{45}{4\eta^4}  \E\left[\left. \left((d_n^l)^2-\frac{\eta^2}{3}\right) \left(\sum\limits_{i=1}^N (d_n^i)^2 \nabla^2_{ii} f(x_n) + 2\sum\limits_{i=1}^{N-1}\sum\limits_{j=i+1}^N d_n^i d_n^j \nabla^2_{ij} f(x_n)\right)\right| \F_n\right] = \nabla^2_{ll} f(x_n), \text{ a.s.}
\end{align*}

\subsection*{Off-diagonal terms in \eqref{eq:h1}:}

We now consider the $(k,l)$th term in \eqref{eq:h1}: Assume w.l.o.g that $k<l$. Then,
\begin{align}
& \dfrac{9}{2\eta^4} \E\left[\left.d_n^k d_n^l   \left(\sum\limits_{i=1}^N (d_n^i)^2 \nabla^2_{ii} f(x_n) + 2\sum\limits_{i=1}^{N-1}\sum\limits_{j=i+1}^N d_n^i d_n^j \nabla^2_{ij} f(x_n)\right)\right| \F_n \right]\nonumber\\
=& \dfrac{9}{2\eta^4} \sum\limits_{i=1}^N \E \left(d_n^k d_n^l (d_n^i)^2 \right)\nabla^2_{ii} f(x_n) + \dfrac{9}{\eta^4}\sum\limits_{i=1}^{N-1}\sum\limits_{j=i+1}^N \E\left(d_n^k d_n^l d_n^i d_n^j\right) \nabla^2_{ij} f(x_n) \label{eq:crossh}\\
= & \nabla^2_{kl} f(x_n).\nonumber
\end{align}
The last equality follows from the fact that the first term in \eqref{eq:crossh} is $0$ since $k\ne l$, while the second term in \eqref{eq:crossh} can be seen to be equal to $\dfrac{9}{\eta^4} \E\left((d_n^k)^2 (d_n^l)^2\right) \nabla^2_{kl} f(x_n) = \nabla^2_{kl} f(x_n)$.
The claim follows for the case of uniform perturbations.\ \\\ \\

\textbf{Case 2: Asymmetric Bernoulli perturbations}

Note that the proof up to \eqref{eq:h1} is independent of the choice of perturbations. The proof differs in the analysis of the diagonal and off-diagonal terms in \eqref{eq:h1}. In the case of asymmetric Bernoulli perturbations, the normalizing scalars in the definition of $M_n$ in \eqref{eq:2rdsa-estimate-ber} are different. 

\subsection*{Diagonal terms in \eqref{eq:h1}}
Let $\phi = \dfrac{(1+\epsilon)(1+(1+\epsilon)^3)}{(2+\epsilon)}$ denote the fourth moment $E (d_n^i)^4$, for any $i=1,\ldots,N$.
An analogue of \eqref{eq:h2} is as follows:
\begin{align}
& \frac{1}{\phi(1- \frac{(1+\epsilon)^2}{\phi})}  \E\left(\left. \left((d_n^l)^2-(1+\epsilon)\right) \left(\sum\limits_{i=1}^N (d_n^i)^2 \nabla^2_{ii} f(x_n) + 2\sum\limits_{i=1}^{N-1}\sum\limits_{j=i+1}^N d_n^i d_n^j \nabla^2_{ij} f(x_n)\right)\right| \F_n\right)\nonumber\\
=& \frac{1}{\phi(1- \frac{(1+\epsilon)^2}{\phi})} \E\left(\left. (d_n^l)^2 \sum\limits_{i=1}^N (d_n^i)^2 \nabla^2_{ii} f(x_n) \right| \F_n\right)  - \frac{(1+\epsilon)}{\phi(1- \frac{(1+\epsilon)^2}{\phi})} \E\left(\left. \sum\limits_{i=1}^N (d_n^i)^2 \nabla^2_{ii} f(x_n) \right| \F_n\right) \label{eq:h2-a}
\end{align}
We have used the fact that the second term in the LHS above is conditionally zero  (see argument below \eqref{eq:h2} for a justification).
The first term on the RHS of \eqref{eq:h2-a} be simplified as follows:
\begin{align}
&\frac{1}{\phi(1- \frac{(1+\epsilon)^2}{\phi})} \E\left(\left. (d_n^l)^2 \sum\limits_{i=1}^N (d_n^i)^2 \nabla^2_{ii} f(x_n) \right| \F_n\right)\nonumber\\
= & \frac{1}{\phi(1- \frac{(1+\epsilon)^2}{\phi})} \E\left((d_n^l)^4 \nabla^2_{ll} f(x_n) + \sum\limits_{i=1,i\ne l}^N (d_n^l)^2(d_n^i)^2 \nabla^2_{ii} f(x_n) \right)\nonumber\\
= & \frac{1}{(1- \frac{(1+\epsilon)^2}{\phi})} \left( \nabla^2_{ll} f(x_n) + \dfrac{(1+\epsilon)^2}{\phi}\sum\limits_{i=1,i\ne l}^N  \nabla^2_{ii} f(x_n) \right).\label{eq:diag1}
\end{align} 
For the second equality above, we have used the fact that $\E[(d_n^l)^4] = \phi$ and $\E[(d_n^l)^2 (d_n^i)^2] = \E[(d_n^l)^2] \E[(d_n^i)^2] = (1+\epsilon)^2$, $\forall l \ne i$.

The second term in \eqref{eq:h2-a} with the conditional expectation and without the negative sign can be simplified as follows:
\begin{align}
\frac{(1+\epsilon)}{\phi(1- \frac{(1+\epsilon)^2}{\phi})} \E\left(\left. \sum\limits_{i=1}^N (d_n^i)^2 \nabla^2_{ii} f(x_n) \right| \F_n\right)
= & \frac{(1+\epsilon)}{\phi(1- \frac{(1+\epsilon)^2}{\phi})} \sum\limits_{i=1}^N \E \left[(d_n^i)^2\right] \nabla^2_{ii} f(x_n) \nonumber\\
= & \frac{(1+\epsilon)^2}{\phi(1- \frac{(1+\epsilon)^2}{\phi})} \sum\limits_{i=1}^N \nabla^2_{ii} f(x_n).\label{eq:diag2}
\end{align} 
Combining \eqref{eq:diag1} and \eqref{eq:diag2}, the correctness of the Hessian estimate follows for the diagonal terms.

\subsection*{Off-diagonal terms in \eqref{eq:h1}}
Consider the $(k,l)$th term in \eqref{eq:h1}, with $k<l$. We obtain
\begin{align}
& \dfrac{1}{2(1+\epsilon)^2} \E\left[\left.d_n^k d_n^l   \left(\sum\limits_{i=1}^N (d_n^i)^2 \nabla^2_{ii} f(x_n) + 2\sum\limits_{i=1}^{N-1}\sum\limits_{j=i+1}^N d_n^i d_n^j \nabla^2_{ij} f(x_n)\right)\right| \F_n \right]\nonumber\\
=& \dfrac{1}{2(1+\epsilon)^2} \sum\limits_{i=1}^N \E \left(d_n^k d_n^l (d_n^i)^2 \right)\nabla^2_{ii} f(x_n) + \dfrac{1}{(1+\epsilon)^2}\sum\limits_{i=1}^{N-1}\sum\limits_{j=i+1}^N \E\left(d_n^k d_n^l d_n^i d_n^j\right) \nabla^2_{ij} f(x_n) \label{eq:crossh-a}\\
= & \nabla^2_{kl} f(x_n).\nonumber
\end{align}
Note that the first term on the RHS of \eqref{eq:crossh-a} equals zero since $k\ne l$.
The claim follows for the case of asymmetric Bernoulli perturbations.
\end{proof}

\begin{theorem}(\textbf{Strong Convergence of parameter})
\label{thm:2rdsa-x}
Assume (C1)-(C8). Then $x_n \rightarrow x^*$ a.s. as $n\rightarrow\infty$, where $x_n$ is given by \eqref{eq:2rdsa}. 
\end{theorem}
\begin{proof}
From Lemma \ref{lemma:1rdsa-bias}, we observe that the gradient estimate $\widehat\nabla f(x_n)$ in \eqref{eq:2rdsa} satisfies:
$$ \E\left[\left.\widehat\nabla f(x_n)\right| \F_n \right]= \nabla f(x_n) + \beta_n,$$
where the bias term $\beta_n$ is such that $\delta_n^{-2} \left\|\beta_n\right\|$ is uniformly bounded for sufficiently large $n$. 
The rest of the proof follows in a manner similar to the proof of Theorem 1a in \cite{spall_adaptive}\footnote{Note that the proof of Theorem \ref{thm:2rdsa-x} does not assume any particular form of the Hessian estimate and only requires assumptions (C1)-(C7), which are similar to those in \cite{spall_adaptive}. The only variation in our case, in comparison to \cite{spall_adaptive}, is the gradient estimation uses an RDSA scheme while \cite{spall_adaptive} uses first-order SPSA.  Thus, only the first step of the proof differs and in our case, Lemma \ref{lemma:1rdsa-bias} controls the bias with the same order as that of SPSA, leading to the final result.}. 
\end{proof}

\begin{theorem}(\textbf{Strong Convergence of Hessian})
\label{thm:2rdsa-H}
Assume (C1)-(C10). Then $\overline H_n \rightarrow \nabla^2 f(x^*)$ a.s. as $n\rightarrow \infty$, where $\overline H_n$ is governed by \eqref{eq:2rdsa-H} and $\widehat H_n$ defined according to either \eqref{eq:2rdsa-estimate-unif} or \eqref{eq:2rdsa-estimate-ber}. 
\end{theorem}
\begin{proof}
We first use a martingale convergence result to show that \\
$
 \frac{1}{n+1} \sum_{m=0}^n \left( \widehat H_m - \E\left[\left. \widehat H_m \right| x_m\right]\right) \rightarrow 0 \text{ a.s.}
$ Next, using Lemma \ref{lemma:2rdsa-bias}, we can conclude that \\$
\frac{1}{n+1} \sum_{m=0}^n \E\left[\left. \widehat H_m \right| x_m\right]
\rightarrow  \nabla^2 f(x^*) \text{ a.s.}
$
and the claim follows. The reader is referred to Appendix \ref{sec:appendix-2rdsa} for the detailed proof.
\end{proof}

We next present a asymptotic normality result for 2RDSA under the following additional assumptions:
\begin{enumerate}[label=(\textbf{C\arabic*}),resume]
\item  For some $\zeta, \alpha_0, \alpha_1>0$ and for all $n$, $\E {\xi_n}^{2+\zeta} \le \alpha_0$, $\E {\xi_n^{\pm}}^{2+\zeta} \le \alpha_0$, $\E f(x_n)^{2+\zeta} \le \alpha_1$ and $\E f(x_n\pm \delta_n d_n)^{2+\zeta} \le \alpha_1$. 
\item The operator $\Upsilon$ is chosen such that $\left\|\Upsilon(\overline H_n) - \overline H_n\right\| \rightarrow 0$ a.s. as $n\rightarrow \infty$.
\end{enumerate}
Assumption (C11) is required to ignore the effects of noise, while (C12) together with Theorem \ref{thm:2rdsa-H} ensures that $\Upsilon(\overline H_n)$ converges to the true Hessian a.s. It is easy to see that the choice suggested in Remark \ref{remark:upsilon} for $\Upsilon$ satisfies (C12).
 
The main result is as follows:
\begin{theorem}(\textbf{Asymptotic Normality})
\label{thm:2rdsa-asymp-normal}
Assume (C1)-(C12) and  that $\nabla^2 f(x^*)^{-1}$ exists.
Let $a_n=a_0/n^\alpha$ and $\delta_n = \delta_0/n^\gamma$, where $a_0,\delta_0 >0$,  $\alpha \in (0,1]$ and $\gamma \ge 1/6$. Let $\beta = \alpha - 2 \gamma$. Let $\E(\xi_n^+ - \xi_n^-)^2 \rightarrow \sigma^2$ as $n\rightarrow \infty$. 
Then, we have
\begin{align}
 n^{\beta/2}(x_n - x^*) \xrightarrow{dist} \N(\mu, \Omega)  \text{ as } n\rightarrow \infty,
\end{align}
where $\N(\mu, \Omega)$ is the multivariate Gaussian distribution with mean $\mu$ and covariance matrix $\Omega$. 
The mean $\mu$ is defined as follows:
$\mu=0$ (an $N$-vector of all zeros) if $\gamma > \alpha/6$ and $\mu = k_\mu (a_0 \delta_0^2 (2a_0-\beta^+)^{-1}\nabla^2 f(x^*)^{-1} T)$ if $\gamma = \alpha/6$, where
\begin{equation*}
\label{eq:mu-2}
 k_\mu =
  \begin{cases}
3.6 &  \text{ for uniform perturbations,}\\
\dfrac{2(1+\epsilon)(1+(1+\epsilon)^3)}{(2+\epsilon)(1+\epsilon)^2}  &  \text{ for asymmetric Bernoulli perturbations.}
  \end{cases}
\end{equation*}
In the above, $T$ and $\beta^+$ are as defined in Theorem \ref{thm:1rdsa-asymp-normal}.
The covariance matrix $\Omega$ is defined as follows:
$$ \Omega= \dfrac{a_0^2\sigma^2}{4\delta_0^2 \rho^2 (8a_0-4\beta^+)} \left(\nabla^2 f(x^*)^{-1}\right)^2.$$
\end{theorem}
\begin{proof}
As in the case of 2SPSA of \cite{spall_adaptive}, we verify conditions (2.2.1)-(2.2.3) of \cite{fabian1968asymptotic} to establish the result and the reader is referred to Appendix \ref{sec:appendix-2rdsa} for  details.  
\end{proof}

%%%%%%%%%%%%%%%%%%%%%%%%%%%%%%%%%%%%%%%%%%%%%%%%%%%%%%%%%%%%%%%%%%%%%%%%%%%%%%%%%%%%%%%%%%%%%%%%%%%%%%%%%%%%%%%%%%%%%%%%%%%%%%%%%%%%%%
\subsection{(Asymptotic) convergence rates}
\label{sec:amse}
Recall from Theorems \ref{thm:1rdsa-asymp-normal} and \ref{thm:2rdsa-asymp-normal} that we set $a_n=a_0/n^\alpha$ and $\delta_n = \delta_0/n^\gamma$, where $a_0,\delta_0 >0$,  $\alpha \in (0,1]$ and $\gamma \ge 1/6$.
Let $\mathcal{AMSE}_{2\mathcal{RDSA-}\textit{Unif}}(a_0,\delta_0)$ and $\mathcal{AMSE}_{2\mathcal{RDSA-}\textit{AsymBer}}(a_0, \delta_0)$ denote the AMSE for the uniform and asymmetric Bernoulli variants of 2RDSA, respectively.
These quantities can be derived using Theorems \ref{thm:1rdsa-asymp-normal} and \ref{thm:2rdsa-asymp-normal} as follows:
\begin{align*}
 \mathcal{AMSE}_{2\mathcal{RDSA-}\textit{Unif}}(a_0,\delta_0) =& \left(\frac{3.6 \delta_0^2 a_0}{2a-\beta} \l (\nabla^2 f(x^*)^{-1} T \r\right)^2 \\
 &+ \frac{a^2}{\delta_0^{2}(2a-\beta)} \text{trace}\left( \nabla^2 f(x^*)^{-1} S \nabla^2 f(x^*)^{-1}\right),\\
  \mathcal{AMSE}_{2\mathcal{RDSA-}\textit{AsymBer}}(a_0,\delta_0) =& \left(\frac{2\beta \delta_0^2 a_0}{(1+\epsilon)^2(2a-\beta)} \l (\nabla^2 f(x^*)^{-1} T \r\right)^2 \\&+ \frac{a^2}{\delta_0^{2}(2a-\beta)} \text{trace}\left( \nabla^2 f(x^*)^{-1} S \nabla^2 f(x^*)^{-1}\right),
\end{align*}
where $T$ is as defined in Theorem \ref{thm:1rdsa-asymp-normal} and $S=\frac{\sigma^2}{4} I$.

Recall from the discussion in Section \ref{sec:amse-1rdsa} that 1RDSA has a problem dependence on the minimum eigenvalue $\lambda_0$ of $\nabla^2 f(x^*)$ for the step-size constant $a_0$. One can get rid of this dependence by using one of the 2RDSA variants and $a_0=1$.
An alternative is to use iterate averaging, whose AMSE can be shown to be: 
\begin{align*}
 \mathcal{AMSE}_{1\mathcal{RDSA-}\textit{Avg}}(\delta_0) =& \left(\frac{3.6 \delta_0^2}{2-\beta} \l (\nabla^2 f(x^*)^{-1} T \r  \right)^2 + \frac{1}{\delta_0^{2}(2-\beta)} \text{trace}\left( \nabla^2 f(x^*)^{-1} S \nabla^2 f(x^*)^{-1}\right).
\end{align*}
Notice that with either variant of 2RDSA one obtains the same rate as  with iterate averaging and both these schemes do not have  dependence on $\lambda_0$. 
Moreover, using arguments similar to \cite{dippon1997weighted} (see expressions (5.2) and (5.3) there), we obtain  
\begin{align*}
\forall \delta_0, \mathcal{AMSE}_{2\mathcal{RDSA-}\textit{Unif}}(1,\delta_0) < 2 \min_{a_0>\beta/(2\lambda_0)} \mathcal{AMSE}_{1\mathcal{RDSA}}(a_0,\delta_0),\\
\forall \delta_0, \mathcal{AMSE}_{2\mathcal{RDSA-}\textit{AsymBer}}(1,\delta_0) < 2 \min_{a_0>\beta/(2\lambda_0)} \mathcal{AMSE}_{1\mathcal{RDSA}}(a_0,\delta_0). 
\end{align*}
Note that the above bound holds for any choice of $\delta_0$. Thus, 2RDSA is a robust scheme, as a wrong choice for $a_0$ would adversely affect the bound for 1RDSA, while 2RDSA has no such dependence on $a_0$.

\begin{remark} 
(\textbf{Iterate averaging}) Only from an ``asymptotic'' convergence rate viewpoint is it optimal to use larger step-sizes and iterate averaging. Finite-sample analysis (Theorem 2.4 in \cite{fathi2013transport}) shows that the initial error (which depends on the starting point of the algorithm) is not forgotten sub-exponentially fast, but at the rate $1/n$, where $n$ is the number of iterations. Thus, the effect of averaging kicks in only after enough iterations have passed and the bulk of the iterates are centered around the optimum. See Section 4.5 in \cite{spall2005introduction} for a detailed discussion on this topic.
\end{remark}

\textbf{Comparing 2RDSA-Unif vs 2SPSA.}\\
Taking the ratio of AMSE of 2RDSA with uniform perturbations to that of 2SPSA, we obtain:
\begin{align}
 \dfrac{\mathcal{AMSE}_{2\mathcal{RDSA-}\textit{Unif}}(1,\delta_0)}{\mathcal{AMSE}_{2\mathcal{SPSA}}(1,\delta_0)} =& \dfrac{3.24 (A) + (B)}{ (A) + (B)}, \text{ where}
\end{align}
\begin{align}
(A)=&\left(\frac{2 \delta_0^2}{2-\beta} \l (\nabla^2 f(x^*)^{-1} T \r  \right)^2, \label{eq:A}\\
(B)=&\frac{1}{\delta_0^{2}(2-\beta)} \text{trace}\left( \nabla^2 f(x^*)^{-1} S \nabla^2 f(x^*)^{-1}\right).\label{eq:B}
\end{align}

However, 2SPSA uses four system simulations per iteration, while 2RDSA-Unif uses only three. So, in order to achieve a given accuracy, the ratio of the number of simulations needed for 2RDSA-Unif (denoted by $\hat n_{\text{2RDSA-Unif}}$) to that for 2SPSA (denoted by $\hat n_{\text{2SPSA}}$) is
\begin{align*}
 \dfrac{\hat n_{2\mathcal{RDSA-}\textit{Unif}}}{\hat n_{2\mathcal{SPSA}}} =&  \dfrac{3}{4} \times\dfrac{\mathcal{AMSE}_{2\mathcal{RDSA-}\textit{Unif}}(1,\delta_0)}{\mathcal{AMSE}_{2\mathcal{SPSA}}(1,\delta_0)}
=  \dfrac{3}{4}\times \dfrac{3.24 (A) + (B)}{ (A) + (B)}
= 1 + \dfrac{5.72 (A) - (B)}{ 4(A) + 4(B)}.
\end{align*}
Thus, if $5.72 (A) - (B) < 0$, 2RDSA-Unif's AMSE is better than 2SPSA. On the other hand, if $5.72 (A) - (B) \approx 0$, 2RDSA-Unif is comparable to 2SPSA and finally, in the case where $5.72 (A) - (B) >0$, 2SPSA is better, but the difference may be minor unless $5.72 (A) >> (B)$, as we have the term  $4(A) + 4(B)$ in the denominator above. Note that the quantities $(A)$ and $(B)$ are problem-dependent, as they require knowledge of $\nabla^2 f(x^*)$ and $T$. 

Unlike the first-order algorithms, one cannot conclude that 2SPSA is better than 2RDSA-Unif even when $\nabla^2 f(x^*)$ and $T$ are known.
2RDSA-Unif uses fewer simulations per iteration, which may tilt the balance in favor of 2RDSA-Unif. 
We next show that asymmetric Bernoulli distributions are a better alternative, as 
they result in an AMSE for 2RDSA schemes that is lower than that for 2SPSA on \textit{all} problem instances.

\textbf{Comparing AMSE of 2RDSA-AsymBer to that of 2SPSA.}\\
Taking the ratio of AMSE of 2RDSA with asymmetric Bernoulli perturbations to that of 2SPSA, we obtain:
\begin{align*}
 \dfrac{\mathcal{AMSE}_{2\mathcal{RDSA-}\textit{AsymBer}}(1,\delta_0)}{\mathcal{AMSE}_{2\mathcal{SPSA}}(1,\delta_0)} =& \dfrac{\frac{\beta^2}{(1+\epsilon)^2} (A) + (B)}{ (A) + (B)}, 
\end{align*}
where $(A)$ and $(B)$ are as defined in \eqref{eq:A}--\eqref{eq:B}.
However, 2SPSA uses four system simulations per iteration, while 2RDSA-AsymBer uses only three. So, in order to achieve a given accuracy, the ratio of the number of simulations needed for 2RDSA-AsymBer (denoted by $\hat n_{\text{2RDSA-AsymBer}}$) to that for 2SPSA (denoted by $\hat n_{\text{2SPSA}}$) is
\begin{align*}
 \dfrac{\hat n_{2\mathcal{RDSA-}\textit{AsymBer}}}{\hat n_{2\mathcal{SPSA}}} =&  \dfrac{3}{4} \times\dfrac{\mathcal{AMSE}_{2\mathcal{RDSA-}\textit{AsymBer}}(1,\delta_0)}{\mathcal{AMSE}_{2\mathcal{SPSA}}(1,\delta_0)}
=  \dfrac{3}{4}\times \dfrac{\frac{\beta^2}{(1+\epsilon)^2}  (A) + (B)}{ (A) + (B)}
\end{align*}
For the sake of illustration, we set $\epsilon=0.01$ in the asymmetric Bernoulli distribution \eqref{eq:det-proj} to obtain
\begin{align*}
  \dfrac{\hat n_{2\mathcal{RDSA-}\textit{AsymBer}}}{\hat n_{2\mathcal{SPSA}}}   =  \dfrac{3.0000057  (A) + 3 (B)}{ 4(A) + 4(B)} < 1.
\end{align*}
Thus, 2RDSA with asymmetric Bernoulli distribution AMSE is clearly better than 2SPSA on \textit{all problem instances}, as $(A)$ and $(B)$ are positive (albeit unknown) quantities. 

%%%%%%%%%%%%%%%%%%%%%%%%%%%%%%%%%%%%%%%%%%%%%%%%%%%%%%%%%%%%%%%%%%%%%%%%%%%%%%%%%%%%%%%%%%%%%%%%%%%%%%%%%%%%%%%%%%%%%%%%%%%%%%%%%%%%%%

%%%%%%%%%%%%%%%%%%%%%%%%%%%%%%%%%%%%%%%%%%%%%%%%%%%%%%%%%%%%%%%%%%%%%%%%%%%%%%%%%%%%%%%%%%%%%%%%%%%%%%%%%%%%%%%%%%%%%%%%%%%%%%%%%%%%%%
\section{Numerical Experiments}
\label{sec:experiments}
\subsection{Setting}
We use two functions, both in $N=10$ dimensions for evaluating our algorithms.
 \begin{description}
 \item[Quadratic function:] 
Let $A$ be such that $NA$ is an upper triangular matrix with each entry one and $b$ is the $N$-dimensional vector of ones. Then, the quadratic objective function is defined as follows: 
\begin{align}
f(x) = x\tr A x + b\tr x,\label{eq:quadratic}
\end{align} 
The optimum $x^*$ for $f$ is such that each coordinate of $x^*$ is $-0.9091$, with $f(x^*) = -4.55$.  
 \item[Fourth-order function:] This is the function used for evaluating the second-order SPSA algorithm in \cite{spall_adaptive} and is given as follows:
\begin{align} 
 f(x) = x\tr A\tr A x + 0.1 \sum_{j=1}^N (Ax)^3_j + 0.01 \sum_{j=1}^N (Ax)^4_j,\label{eq:4thorder}
 \end{align} 
 where $A$ is the same as that in the case of quadratic loss.  The optimum $x^*=0$ with $f(x^*)=0$.
 \end{description}
%We use the following quadratic loss function in $N=10$ dimensions for evaluating our algorithms:
%$$ f(x) = x\tr A x + b\tr x,$$
%where $A$ is such that $NA$ is an upper triangular matrix with each entry one and $b$ is the $N$-dimensional vector of ones. 
For any $x$, the noise is $[x\tr, 1]z$, where $z$ is distributed as a multivariate Gaussian distribution in $11$ dimensions with mean $0$ and covariance  $\sigma^2 I_{11\times11}$. 
As remarked in \cite{spall_adaptive}, the motivation for this noise structure is to have most of the noise components depend on the iterate $x_n$ and also a component $z$ that ensures that the variance is at least $\sigma^2$.  

\subsection{Implementation}
We implement the following algorithms\footnote{The implementation is available at \url{https://github.com/prashla/RDSA/archive/master.zip}.}:
\begin{description}
\item[First-order:] This class includes the RDSA schemes with uniform and asymmetric Bernoulli distributions - 1RDSA-Unif and 1RDSA-AsymBer, respectively, and regular SPSA with Bernoulli perturbations - 1SPSA.
\item[Second-order:] This class includes 2RDSA-Unif and 2RDSA-AsymBer - the second-order RDSA schemes with uniform and asymmetric Bernoulli distributions, respectively, and also regular second-order SPSA with Bernoulli perturbations - 2SPSA.
\end{description}

For 1RDSA and 1SPSA, we set $\delta_n = 1.9/n^{0.101}$ and $a_n = 1/(n+50)$. For 2RDSA and 2SPSA, we set $\delta_n = 3.8/n^{0.101}$ and $a_n = 1/n^{0.6}$. These choices are motivated by standard guidelines - see \cite{spall2005introduction}. For uniform perturbation variants, we set the distribution parameter $\eta=1$ and for the asymmetric Bernoulli variants, we set the distribution parameter $\epsilon$ as follows: $\epsilon=0.0001$ for 1RDSA-AsymBer and $\epsilon=1$ for 2RDSA-AsymBer. These choices are motivated by a sensitivity study with different choices for $\epsilon$ - see Tables \ref{tab:1rdsa-epsilon}--\ref{tab:2rdsa-epsilon} in Appendix \ref{sec:appendix-results}.
For all the algorithms, the initial point $x_0$ is the $N=10$-dimensional vector of ones. To keep the iterates stable, each coordinate of the parameter $\theta$ is projected onto the set $[-2.048, 2.047]$.  
All results are averages over $1000$ replications.

We use normalized mean square error (NMSE) as the performance metric for comparing algorithms. This quantity is defined as follows:
$$ \text{NMSE} = \dfrac{\l x_{n_\text{end}} - x^* \r^2}{\l x_0 - x^*\r^2},$$ 
where $x_{n_\text{end}}$ is the algorithm iterate at the end of the simulation. Note, $n_\text{end}$ is algorithm-specific and a function of the number of measurements. For instance, with $2000$ measurements, $n=1000$ for both 1SPSA and both 1RDSA variants, as they use two measurements per iteration. On the other hand, for 2SPSA and both 2RDSA variants, an initial $20\%$ of the measurements were used up by 1SPSA/1RDSA and the resulting iterates were used to initialize the corresponding second-order method. Thus, with $2000$ measurements available, the initial $400$ measurements are used for 1SPSA/1RDSA and the remaining $1600$ are used up by 2SPSA/2RDSA. 
This results in $n_\text{end}$ of $1600/4=400$ for 2SPSA and $1600/3\approx 533$ for 2RDSA algorithms. Note 
that the difference here is due to the fact that 2RDSA uses $3$ simulations per iteration, while 2SPSA needs $4$. 

%{\protect \footnote{For 1RDSA-Unif/2RDSA-Unif, $\eta=1$, while $\epsilon=0.0001$ for 1RDSA-AsymBer and $\epsilon=1$ for 2RDSA-AsymBer.}}
\begin{table}
\centering
 \caption{NMSE for quadratic objective \eqref{eq:quadratic} and noise parameter $\sigma=0.001$:\\ standard error from $1000$ replications shown after $\pm$}
\label{tab:mse-1}
\small
\begin{tabular}{|c|c|c|c|}
\toprule
%\rowcolor{gray!20}
\multicolumn{4}{|c|}{\multirow{2}{*}{\textbf{First-order Algorithms}}}\\[1.7em]
%&&&\\
\midrule
 \textbf{No. of function} & \multirow{2}{*}{\textbf{1SPSA}} & \multirow{2}{*}{\textbf{1RDSA-Unif}} & \multirow{2}{*}{\textbf{1RDSA-AsymBer}} \\
 \textbf{measurements} & & & \\
 \midrule
 $1000$ & $\bm{4.15 \times 10^{-2} \pm 5.15\times 10^{-4}}$ & $4.53\times 10^{-2} \pm 5.72\times 10^{-4}$ & $4.18\times 10^{-2} \pm 5.41\times 10^{-4}$\\
&&&\\
 $2000$ & $3.42\times 10^{-2} \pm 4.68\times 10^{-4}$ & $3.67\times 10^{-2} \pm 5.28\times 10^{-4}$ & $\bm{3.38\times 10^{-2} \pm 4.84\times 10^{-4}}$\\
 \bottomrule
 %%%%%%%%%%%%2nd ordr
% \rowcolor{gray!20}
\multicolumn{4}{|c|}{\multirow{2}{*}{\textbf{Second-order Algorithms}}}\\[1.7em]
\midrule
  \textbf{No. of function} & \multirow{2}{*}{\textbf{2SPSA}} & \multirow{2}{*}{\textbf{2RDSA-Unif}} & \multirow{2}{*}{\textbf{2RDSA-AsymBer}} \\
 \textbf{measurements} & & & \\
 \midrule
 $1000$ & $1.05\times 10^{-3} \pm 2.25\times 10^{-5}$  & $9.61\times 10^{-5} \pm 2.48\times 10^{-6}$ & $\bm{8.39\times 10^{-5} \pm 2.25 \times 10^{-6}}$ \\
&&&\\
 $2000$ & $3.60\times 10^{-6} \pm 7.62\times 10^{-8}$  & $4.48\times 10^{-6} \pm 6.61\times 10^{-8}$ & $\bm{2.24\times 10^{-6} \pm 3.35 \times 10^{-8}}$\\
 \bottomrule 
\end{tabular}
\end{table}

%%%%%%%%%%%%%%% zero noise %%%%%%%%%%%%%%%%%%%%%%%%%
\begin{table}
\centering
\small
 \caption{NMSE for quadratic objective \eqref{eq:quadratic} and noise parameter $\sigma=0$:\\ standard error from $1000$ replications shown after $\pm$}
\label{tab:mse-01}
\begin{tabular}{|c|c|c|c|}
\toprule
%\rowcolor{gray!20}
\multicolumn{4}{|c|}{\multirow{2}{*}{\textbf{First-order Algorithms}}}\\[1.7em]
%&&&\\
\midrule
 \textbf{No. of function} & \multirow{2}{*}{\textbf{1SPSA}} & \multirow{2}{*}{\textbf{1RDSA-Unif}} & \multirow{2}{*}{\textbf{1RDSA-AsymBer}} \\
 \textbf{measurements} & & & \\
 \midrule
 $1000$ & $\bm{4.15 \times 10^{-2} \pm 5.15\times 10^{-4}}$ & $4.53\times 10^{-2} \pm 5.72\times 10^{-4}$ & $4.18\times 10^{-2} \pm 5.41\times 10^{-4}$\\
&&&\\
 $2000$ & $3.42\times 10^{-2} \pm 4.68\times 10^{-4}$ & $3.67\times 10^{-2} \pm 5.28\times 10^{-4}$ & $\bm{3.37\times 10^{-2} \pm 4.87\times 10^{-4}}$\\
 \bottomrule
 %%%%%%%%%%%%2nd ordr
%\rowcolor{gray!20}
\multicolumn{4}{|c|}{\multirow{2}{*}{\textbf{Second-order Algorithms}}}\\[1.7em]
\midrule
  \textbf{No. of function} & \multirow{2}{*}{\textbf{2SPSA}} & \multirow{2}{*}{\textbf{2RDSA-Unif}} & \multirow{2}{*}{\textbf{2RDSA-AsymBer}} \\
 \textbf{measurements} & & & \\
 \midrule
 $1000$ & $7.57\times 10^{-4} \pm 1.59\times 10^{-5}$  & $9.34\times 10^{-5} \pm 2.49\times 10^{-6}$ & $\bm{8.27\times 10^{-5} \pm 2.25\times 10^{-6}}$ \\
&&&\\
 $2000$ & $6.77\times 10^{-7} \pm 2.78\times 10^{-8}$  & $2.42\times 10^{-9} \pm 1.11\times 10^{-10}$ & $\bm{2.90\times 10^{-9} \pm 1.41\times 10^{-10}}$\\
 \bottomrule 
\end{tabular}
\end{table}

\subsection{Results: Quadratic objective}
Tables \ref{tab:mse-1}--\ref{tab:mse-01} present the normalized mean square error (NMSE) for both first and second-order algorithms with quadratic objective \eqref{eq:quadratic} and noise parameter $\sigma$ set to $0.001$ and $0$. 

\textit{\textbf{Observation 1:} Among first-order schemes, 1RDSA-AsymBer performs on par with 1SPSA, while 1RDSA-Unif is sub-par.}

The NMSE  of 1RDSA-AsymBer is comparable to that of 1SPSA, while 1RDSA-Unif results in a higher NMSE. This is consistent with the asymptotic rate results discussed earlier in Section \ref{sec:amse-1rdsa}.

\textit{\textbf{Observation 2:} Second-order schemes outperform their first-order counterparts, and  2RDSA-AsymBer performs best in this class.}

The first part of the observation is consistent with earlier results for the low noise regime (i.e., $\sigma=0.01$), for instance, see \cite{spall_adaptive}. Further, the gains of using second-order schemes are more noticeable in the zero-noise regime (see Table \ref{tab:mse-01}). Moreover, 2RDSA-AsymBer results in the best NMSE. In fact, running 2RDSA-AsymBer for $400$ iterations, which was the number used for 2SPSA with $2000$ measurements available, the resulting NMSE values was found to be $2.34\times 10^{-6}$, which is  better than the corresponding $400$-iteration result of $3.60\times10^{-6}$ for 2SPSA  (see Table \ref{tab:mse-1}) , while using only $75\%$ as many simulations. 

\subsection{Results: Fourth-order objective}
Table \ref{tab:mse-4th} presents results similar to those in Table \ref{tab:mse-1} for the fourth-order objective function \eqref{eq:4thorder} with the noise parameter $\sigma$ set to $0.001$. In addition, we also present the normalized function values in Table \ref{tab:mse-4thf}. The normalized function value is defined as the ratio $f(x_{n_\text{end}})/f(x_0)$. Considering that the fourth-order objective is more difficult to optimize in comparison to the quadratic one, we run all algorithms with a simulation budget of $10000$ function evaluations. From the results in Tables \ref{tab:mse-4thf}--\ref{tab:mse-4th}, one can draw conclusions similar to that in observations 1 and 2 above, except that 2RDSA-Unif shows the best performance among second-order schemes.  

%%%%%%%%%%%%%%%%%%%%%% 4th order %%%%%%%%%%%%%%%%%%%%%%%%%%%%%%%
\begin{table}
\centering
\small
 \caption{NMSE for fourth-order objective \eqref{eq:4thorder} and noise parameter $\sigma=0.001$:\\ standard error from $1000$ replications shown after $\pm$}
\label{tab:mse-4th}
\begin{tabular}{|c|c|c|c|}
\toprule
%\rowcolor{gray!20}
\multicolumn{4}{|c|}{\multirow{2}{*}{\textbf{First-order Algorithms}}}\\[1.7em]
%&&&\\
\midrule
 \textbf{No. of function} & \multirow{2}{*}{\textbf{1SPSA}} & \multirow{2}{*}{\textbf{1RDSA-Unif}} & \multirow{2}{*}{\textbf{1RDSA-AsymBer}} \\
 \textbf{measurements} & & & \\
 \midrule
 $2000$ & $1.37\times 10^{-1} \pm 1.39 \times 10^{-3}$ & $1.38\times 10^{-1} \pm 1.33 \times 10^{-3}$ & $\bm{1.35 \times 10^{-1} \pm 1.36 \times 10^{-3}}$\\
&&&\\
 $10000$ & $\bm{1.14\times 10^{-1} \pm 1.14\times 10^{-3}}$ & $1.18\times 10^{-1} \pm 1,23\times 10^{-3}$ & $\bm{1.14\times 10^{-1} \pm 1.23\times 10^{-3}}$\\
 \bottomrule
 %%%%%%%%%%%%2nd ordr
%\rowcolor{gray!20}
\multicolumn{4}{|c|}{\multirow{2}{*}{\textbf{Second-order Algorithms}}}\\[1.7em]
\midrule
  \textbf{No. of function} & \multirow{2}{*}{\textbf{2SPSA}} & \multirow{2}{*}{\textbf{2RDSA-Unif}} & \multirow{2}{*}{\textbf{2RDSA-AsymBer}} \\
 \textbf{measurements} & & & \\
 \midrule
 $2000$ & $3.2\times 10^{-2} \pm 5.38\times 10^{-4}$  & $\bm{1.48\times 10^{-2} \pm 2.64\times 10^{-4}}$ & $4.89\times 10^{-2} \pm 9.01\times 10^{-4}$ \\
&&&\\
 $10000$ & $1.01 \times 10^{-2}  \pm  1.96 \times 10^{-4}$  & $\bm{1.74\times 10^{-3} \pm 3.65\times 10^{-5}}$ & $6.45\times 10^{-2} \pm 1.48 \times 10^{-3}$\\
 \bottomrule 
\end{tabular}
\end{table}

\begin{table}
\centering
\small
 \caption{Normalized function values for fourth-order objective \eqref{eq:4thorder} and noise parameter $\sigma=0.001$: standard error from $1000$ replications shown after $\pm$}
\label{tab:mse-4thf}
\begin{tabular}{|c|c|c|c|}
\toprule
%\rowcolor{gray!20}
\multicolumn{4}{|c|}{\multirow{2}{*}{\textbf{First-order Algorithms}}}\\[1.7em]
%&&&\\
\midrule
 \textbf{No. of function} & \multirow{2}{*}{\textbf{1SPSA}} & \multirow{2}{*}{\textbf{1RDSA-Unif}} & \multirow{2}{*}{\textbf{1RDSA-AsymBer}} \\
 \textbf{measurements} & & & \\
 \midrule
 $2000$ & $9.8\times 10^{-3} \pm 1.01\times 10^{-4}$ & $1.1\times 10^{-2} \pm 1,01\times 10^{-4}$ & $\bm{9.6\times 10^{-3} \pm 1.01\times 10^{-3}}$\\
&&&\\
 $10000$ & $\bm{6.1\times 10^{-3} \pm 6.96\times 10^{-5}}$ & $6.3\times 10^{-3} \pm 7,27\times 10^{-5}$ & $\bm{6.1\times 10^{-3} \pm 7.27\times 10^{-5}}$\\
 \bottomrule
 %%%%%%%%%%%%2nd ordr
%\rowcolor{gray!20}
\multicolumn{4}{|c|}{\multirow{2}{*}{\textbf{Second-order Algorithms}}}\\[1.7em]
\midrule
  \textbf{No. of function} & \multirow{2}{*}{\textbf{2SPSA}} & \multirow{2}{*}{\textbf{2RDSA-Unif}} & \multirow{2}{*}{\textbf{2RDSA-AsymBer}} \\
 \textbf{measurements} & & & \\
 \midrule
 $2000$ & $2.55\times 10^{-3} \pm 3.35\times 10^{-5}$  & $\bm{2.17\times 10^{-4} \pm 1.66\times 10^{-5}}$ & $1.97\times 10^{-3} \pm 1.83 \times 10^{-4}$ \\
&&&\\
 $10000$ & $7.62 \times 10^{-4} \pm 1.1\times 10^{-5}$  & $\bm{4.41\times 10^{-5} \pm 4.42\times 10^{-6}}$ & $1.54\times 10^{-3} \pm 1.45 \times 10^{-4}$\\
 \bottomrule 
\end{tabular}
\end{table}

%%%%%%%%%%%%%%%%%%%%%%%%%%%%%%%%%%%%%%%%%%%%%%%%%%%%%%%%%%%%%%%%%%%%%%%%%%%%%%%%%%%%%%%%%%%%%%%%%%%%
\begin{remark}\textbf{\textit{(Enhanced 2SPSA)}}
In \cite{spall-jacobian}, enhancements to the 2SPSA algorithm incorporated adaptive feedback and weighting to improve Hessian estimates. However, preliminary numerical experiments that we conducted for enhanced 2SPSA with the parameters  recommended in \cite{spall-jacobian} indicate that the benefits of such a scheme kick in only after a large number of iterations. For instance, for the fourth-order objective function \eqref{eq:4thorder}, running the enhanced 2SPSA algorithm resulted in a high NMSE and the latter became comparable to that of 2SPSA after increasing the simulation budget of $2000$.  

Finally, the numerical results presented in Tables \ref{tab:mse-1}--\ref{tab:mse-4thf} make a fair comparison in the sense that, except the perturbations every other parameter (e.g.,, step-sizes $a_n$, perturbation constants $\delta_n$, initial point $x_0$) is kept constant across algorithms in each class (first/second-order). The results demonstrate that it is indeed advantageous to use uniform/asymmetric Bernoulli perturbations. It would be interesting future work to enhance 2RDSA schemes to improve the Hessian estimates along the lines of \cite{spall-jacobian} and then numerically compare the performance of enhanced 2RDSA schemes with that of enhanced 2SPSA.
\end{remark}  
%%%%%%%%%%%%%%%%%%%%%%%%%%%%%%%%%%
%%%%%%%%%%%%%%%%%%%%%%%%%%%%%%%%%%%%%%%%%%%%%%%%%%%%%%%%%%%%%%%%%%%%%%%%%%%%%%%%%%%%%%%%%%%%%%%%%%%%%%%%%%%%%%%%%%%%%%%%%%%%%%%%%%%%%%
\section{Conclusions}
\label{sec:conclusions}
We considered a general problem of optimization under noisy
observations and presented the first adaptive random directions
Newton algorithm. Two sets of i.i.d. random perturbations were analyzed: symmetric uniformly distributed and asymmetric Bernoulli distributed. In addition, we also presented a simple gradient
search scheme using two sets of perturbations. While
our gradient search scheme requires the same number of perturbations
and system simulations per iteration as the simultaneous perturbation
gradient scheme of \cite{spall}, our Newton
scheme only requires half the number of perturbations and three-fourths
the number of simulations as compared to the simultaneous
perturbation Newton algorithm of \cite{spall_adaptive}.
We proved the convergence of our algorithms and analyzed their rates
of convergence using the asymptotic mean square error (AMSE). 
From this analysis, we concluded that the asymmetric Bernoulli perturbation variants exhibit the best AMSE for both first- and second-order RDSA schemes. 
% We observed from asymptotic mean square analysis
% and numerical comparisons that our algorithms are computationally
% efficient. 
% In particular, our Newton algorithm results in AMSE 
% that is close to or better in some cases than the algorithm in \cite{spall_adaptive}. 
Furthermore,
our numerical experiments show that our Newton algorithm requires only 75\% of the number of function evaluations
as required by the Newton algorithm of \cite{spall_adaptive} while providing the same accuracy levels as the latter algorithm.

As future work, we outline two possible directions.
First, it would be interesting to use the approach of \cite{spall-jacobian} to arrive at an adaptive scheme that incorporates a feedback term to improve the quality of the Hessian estimate. 
Second, it would be of interest to extend our algorithms to
scenarios where the noise random variables form a parameterized
Markov process and to develop multiscale algorithms in this
setting for long-run average or infinite horizon discounted costs.
Such algorithms will be of relevance in the context
of reinforcement learning, for instance, as actor-critic algorithms.%%%%%%%%%%%%%%%%%%%%%%%%%%%%%%%%%%%%%%%%%%%%%%%%%%%%%%%%%%%%%%%%%%%%%%%%%%%%%%%%%%%%%%%%%%%%%%%%%%%%%%%%%%%%%%%%%%%%%%%%%%%%%%%%%%%%%%
% \bibliographystyle{plainnat}
%%%%%%%%%%%%%%%%%%%%%%%%%%%%%%%%%%%%%%%%%%%%%%%%%%%%%%%%%%%%%%%%%%%%%%%%%%%%%%%%%%%%%%%%%%%%%%%%%%%%%%%%%%%%%%%%%%%%%%%%%%%%%%%%%%%%%%
%%%%%%%%%%%%%%%%%%%%%%%%%%%%%%%%%%%%%%%%%%%%%%%%%%%%%%%%%%%%%%%%%%%%%%%%%%%%%%%%%%%%%%%%%%%%%%%%%%%%%%%%%%%%%%%%%%%%%%%%%%%%%%%%%%%%%%
\section*{Appendix}
\input{appendix}

%%%%%%%%%%%%%%%%%%%%%%%%%%%%%%%%%%%%%%%%%%%%%%%%%%%%%%%%%%%%%%%%%%%%%%%%%%%%%%%%%%%%%%%%%%%%%%%%%%%%%%%%%%%%%%%%%%%%%%%%%%%%%%%%%%%%%%
\bibliographystyle{plainnat}
\bibliography{references}
\end{document}

%% file: appendix.tex
\appendix
%%%%%%%%%%%%%%%%%%%%%%%%%%%%%%%%%%%%%%%%%%%%%%%%%%%%%%%%%%%%%%%%%%%%%%%%%%%%%%%%%%%%%%%%%%%%%%%%%%%%%%%%%%%%%%%%%%%%%%%%%%%%%%%
\section{Proofs for 1RDSA}
\label{sec:appendix-1rdsa}
%%%%%%%%%%%%%%%%%%%%%%%%%%%%%%%%%%%%%%%%%%%%%%%%%%%%%%%%%%%%%%%%%%%%%%%%%%%%%%%%%%%%%%%%%%%%%%%%%%%%%%%%%%%%%%%%%%%%%%%%%%%%%%%
\subsection*{Proof of Theorem \ref{thm:1rdsa-strong-conv}}

\begin{proof}
We first rewrite the update rule \eqref{eq:1rdsa} as follows:
\begin{align}
x_{n+1} = x_n - a_n(\nabla f(x_n) + \eta_n + \beta_n),
\label{eq:1rdsa-equiv}
\end{align}
where $\eta_n = \widehat \nabla f(x_n) - \E(\widehat \nabla f(x_n) \mid \F_n)$ is a martingale difference error term and $\beta_n = \E(\widehat \nabla f(x_n) \mid \F_n) - \nabla f(x_n)$ is the bias in the gradient estimate.
Convergence of \eqref{eq:1rdsa-equiv} can be inferred from Theorem 2.3.1 on pp. 39 of \cite{kushcla}, provided we verify that the assumptions A2.2.1 to A2.2.3 and A2.2.4'' of \cite{kushcla} are satisfied. We mention these assumptions as (B1)-(B4) below.
\begin{enumerate}[label={\textbf(B\arabic*)}]
\item $\nabla f$ is a continuous $\R^N$-valued function.

\item  The sequence $\beta_n,n\geq 0$ is almost surely bounded with
$\beta_n \rightarrow 0$ almost surely as $n\rightarrow \infty$.

\item The step-sizes $a_n,n\geq 0$ satisfy
\[  a(n)\rightarrow 0 \mbox{ as }n\rightarrow\infty \text{ and } \sum_n a_n=\infty.\]

\item $\{\eta_n, n\ge 0\}$ is a sequence such that for any $\epsilon>0$,
\[ \lim_{n\rightarrow\infty} P\left( \sup_{m\geq n}  \left\|
\sum_{i=n}^{m} a_i \eta_i\right\| \geq \epsilon \right) = 0. \]
\end{enumerate} 
The above assumptions can be verified for \eqref{eq:1rdsa-equiv} as follows:
\begin{itemize}
\item (A1) implies (B1).
\item (A5) together with \eqref{eq:l2} in the proof of Lemma \ref{lemma:1rdsa-bias} imply that the bias $\beta_n$ is almost surely bounded. Further, Lemma \ref{lemma:1rdsa-bias} implies that $\beta_n$ is of the order $O(\delta_n^2)$ and since $\delta_n \rightarrow 0$ as $n\rightarrow \infty$ (see (A5)), we have that $\beta_n \rightarrow 0$. Thus, (B2) is satisfied.
\item (A5) implies (B3).
\item We now verify (B4) using arguments similar to those used in Chapter 7.3 of \cite{spall2005introduction}:
We first recall a martingale inequality attributed to Doob (also given as (2.1.7) on pp. 27 of \cite{kushcla}):
\begin{align}
P\left( \sup_{m\geq 0}   \left\|W_m\right\| \geq \epsilon \right) \le \dfrac{1}{\epsilon^2} \lim_{m\rightarrow \infty} \E \left\|W_m\right\|^2. 
\end{align}
We apply the above inequality in our setting to the martingale sequence $\{W_n\}$, where  $W_n := \sum_{i=0}^{n-1} a_i \eta_i$, $n\ge 1$, to obtain
\begin{align}
P\left( \sup_{m\geq n}   \left\|\sum_{i=n}^{m} a_i \eta_i\right\| \geq \epsilon \right) \le \dfrac{1}{\epsilon^2} \E \left\|
\sum_{i=n}^{\infty} a_i \eta_i\right\|^2 = \dfrac{1}{\epsilon^2} \sum_{i=n}^{\infty} a_i^2 \E\left\| \eta_i\right\|^2. \label{eq:b4}
\end{align}
The last equality above follows by observing that, for $m < n$, $\E(\eta_m \eta_n) = \E(\eta_m \E(\eta_n\mid \F_n))=0$.

Using the identity $\E\left\|X -  E[X\mid\F_n]\right\|^2 \le \E \left\|X\right\|^2$ for any random variable $X$, we bound $\E\left\| \eta_n\right\|^2$ as follows:
\begin{align}
\E\left\| \eta_n\right\|^2 \le& N \E \left(\eta^i_n\right)^2 \label{eq:mi} \text{ for some } i\in \{1,\ldots,N\}\\
= & \dfrac{N}{4\delta_n^2} \E\left(d_n^i(y_n^+ -y_n^-)\right)^2 \nonumber\\
\le & \dfrac{N}{4\delta_n^2} \left(\left(\E\left(d_n^i y_n^+ \right)^2\right)^{1/2}
+ \left(\E\left(d_n^i y_n^-\right)^2\right)^{1/2}\right)^2 \label{eq:minko}\\
\le &\frac{N}{4\delta_n^2} \left[ \E\left((d_n^i)^{2+2\rho_1}\right) \right]^{\frac{1}{1+\rho1}} 
\left(\left[\E\left[ (y_n^+)\right]^{2+2\rho_2}\right]^{\frac{1}{1+\rho_2}} +
\left[\E\left[ (y_n^-)\right]^{2+2\rho_2}\right]^{\frac{1}{1+\rho_2}}\right)\label{eq:holder}\\
\le & \frac{C}{\delta_n^2}, \text{ for some } C< \infty. \label{eq:h3}
\end{align}
The inequality in \eqref{eq:minko} follows by the fact that $\E (X+Y)^2 \le \left( (\E X^2)^{1/2} + (\E Y^2)^{1/2}\right)^2$.
The inequality in \eqref{eq:holder} uses Holder's inequality, with $\rho_1, \rho_2>0$ satisfying $\frac{1}{1+\rho_1} + \frac{1}{1+\rho_2}=1$. 
The inequality in \eqref{eq:h2} follows from (A3) and the fact that the perturbations $d_n$ have finite moments. 

Plugging \eqref{eq:h2} into \eqref{eq:b4}, we obtain
\begin{align*}
 \lim_{n\rightarrow\infty} P\left( \sup_{m\geq n}  \left\|
\sum_{i=n}^{m} a_i \eta_i\right\| \geq \epsilon \right) \le \dfrac{C}{\epsilon^2} \lim_{n\rightarrow\infty} \sum_{i=n}^{\infty}  \frac{a_i^2}{\delta_i^2} =0.
\end{align*}
The equality above follows from the the fact that $\sum_n \left(\frac{a_n}{\delta_n}\right)^2 <\infty$ (see (A5)). 
\end{itemize}
The claim follows from Theorem 2.3.1 on pp. 39 of \cite{kushcla}.
\end{proof}
%%%%%%%%%%%%%%%%%%%%%%%%%%%%%%%%%%%%%%%%%%%%%%%%%%%%%%%%%%%%%%%%%%%%%%%%%%%%%%%%%%%%%%%%%%%%%%%%%%%%%%%%%%%%%%%%%%%%%%%%%%%%%%%
%%%%%%%%%%%%%%%%%%%%%%%%%%%%%%%%%%%%%%%%%%%%%%%%%%%%%%%%%%%%%%%%%%%%%%%%%%%%%%%%%%%%%%%%%%%%%%%%%%%%%%%%%%%%%%%%%%%%%%%%%%%%%%%
% \subsection*{Proof of Theorem \ref{thm:1rdsa-asymp-normal}}
% 
% \begin{proof}
% Follows from Proposition 1 of \cite{chin1997comparative} after observing that $\dfrac{3}{\eta^2} \E[d_n d_n\tr] = I$ and \\$\dfrac{9}{\eta^4} \E[(d^i_n)^4] = 1.8$ for any $i=1,\ldots,N$.
% \end{proof}

%%%%%%%%%%%%%%%%%%%%%%%%%%%%%%%%%%%%%%%%%%%%%%%%%%%%%%%%%%%%%%%%%%%%%%%%%%%%%%%%%%%%%%%%%%%%%%%%%%%%%%%%%%%%%%%%%%%%%%%%%%%%%%%
%%%%%%%%%%%%%%%%%%%%%%%%%%%%%%%%%%%%%%%%%%%%%%%%%%%%%%%%%%%%%%%%%%%%%%%%%%%%%%%%%%%%%%%%%%%%%%%%%%%%%%%%%%%%%%%%%%%%%%%%%%%%%%%
%%%%%%%%%%%%%%%%%%%%%%%%%%%%%%%%%%%%%%%%%%%%%%%%%%%%%%%%%%%%%%%%%%%%%%%%%%%%%%%%%%%%%%%%%%%%%%%%%%%%%%%%%%%%%%%%%%%%%%%%%%%%%%%
\section{Proofs for 2RDSA}
\label{sec:appendix-2rdsa}
\subsection*{Proof of Theorem \ref{thm:2rdsa-H}}

\begin{proof}
The proof proceeds in exactly the same manner manner as the proof of Theorem 2a in \cite{spall_adaptive}. For the sake of completeness, we sketch below the main arguments involved in the proof. 

% \subsection*{Step 1: Martingale convergence theorem}
Let $W_m = \widehat H_m - \E\left[\left. \widehat H_m \right| x_m\right]$. Then, we know that $\E W_m =0$. In addition, we have $\sum_m \frac{\E \left\|W_m\right\|^2}{m^2} < \infty$. The latter follows by first observing that $\E\left[ \delta_m^2 \left\|\widehat H_m\right\|^2 \right] < \infty, \forall m$ uniformly as a consequence of (C9) and then coupling this fact with (C8).
Now, applying a martingale convergence result from p. 397 of \cite{lahaprobability} to $W_m$, we obtain
\begin{align}
 \dfrac{1}{n+1} \sum_{m=0}^n \left(\widehat H_m - \E\left[\left. \widehat H_m \right| x_m\right]\right) \rightarrow 0 \text{ a.s.}
\label{eq:h-cesaro}
\end{align}

% \subsection*{Step 2: Biased Hessian estimate}
From Proposition \ref{lemma:2rdsa-bias}, we know that $\E\left[\left. \widehat H_n \right| x_n \right] = \nabla^2 f(x_n) + O(\delta_n^2)$.
\begin{align*}
 &\dfrac{1}{n+1} \sum_{m=0}^n \E\left[\left. \widehat H_m \right| x_m\right]
=  \dfrac{1}{n+1} \sum_{m=0}^n \left(\nabla^2 f(x_m) + O(\delta_m^2)\right)
\rightarrow  \nabla^2 f(x^*) \text{ a.s.}
\end{align*}
The final step above follows from the fact that the Hessian is continuous near $x_n$ and Theorem \ref{thm:2rdsa-x} which implies $x_n$ converges almost surely to $x^*$.
Thus, we obtain
\begin{align*}
 \dfrac{1}{n+1} \sum_{m=0}^n \widehat H_m \rightarrow  \nabla^2 f(x^*) \text{ a.s.}
\end{align*}
and the claim follows by observing that $\overline H_m =   \dfrac{1}{n+1} \sum_{m=0}^n \widehat H_m$.
\end{proof}

%%%%%%%%%%%%%%%%%%%%%%%%%%%%%%%%%%%%%%%%%%%%%%%%%%%%%%%%%%%%%%%%%%%%%%%%%%%%%%%%%%%%%%%%%%%%%%%%%%%%%%%%%%%%%%%%%%%%%%%%%%%%%%%
%%%%%%%%%%%%%%%%%%%%%%%%%%%%%%%%%%%%%%%%%%%%%%%%%%%%%%%%%%%%%%%%%%%%%%%%%%%%%%%%%%%%%%%%%%%%%%%%%%%%%%%%%%%%%%%%%%%%%%%%%%%%%%%
\subsection*{Proof of Theorem \ref{thm:2rdsa-asymp-normal}}

\begin{proof}
We use the well-known result for establishing asymptotic normality of stochastic approximation schemes from  \cite{fabian1968asymptotic}. 
As in the case of SPSA-based algorithms (cf. \cite{spall}, \cite{spall_adaptive}), for 1RDSA, it can be shown that, for sufficiently large $n$, there exists a $\bar x_n$ on the line segment that connects $x_n$ and $x^*$, such that the following holds:
\begin{align*}
\E[\widehat \nabla f(x_n) \mid x_n] = \nabla^2 f(\bar x_n)(x_n - x^*) + \beta_n,
\end{align*}
where $\beta_n = \E(\widehat \nabla f(x_n) \mid x_n) - \nabla f(x_n)$ is the bias in the gradient estimate. 
Next, we write the estimation error $x_{n+1} - x^*$ in a form that is amenable for applying the result from \cite{fabian1968asymptotic}, as follows:
\begin{align}
 x_{n+1} - x^* = (I- n^{-\alpha} \Gamma_n)(x_n - x^*) + n^{-(\alpha+\beta)/2} \Phi_n V_n + n^{(\alpha -\beta)/2}\Upsilon(\overline H_n)^{-1}T_n,
\end{align}
where $\Gamma_n = a_0 \Gamma(\overline H_n)^{-1} \nabla^2 f(\bar x_n)$, $\Phi_n = -a_0 \Gamma(\overline H_n)^{-1}$, $V_n = n^{-\gamma}(\widehat \nabla f(x_n) - \E(\widehat\nabla f(x_n)\mid x_n))$ and $T_n=-a_0 n ^{\beta/2} \beta_n$. 
The above recursion is similar to that for 2SPSA of \cite{spall_adaptive}, except that we estimate the Hessian and gradients using RDSA and not SPSA. 

For establishing the main claim, one needs to verify conditions (2.2.1) to (2.2.3) in Theorem 2.2 of \cite{fabian1968asymptotic}. This can be done as follows:
\begin{itemize}
 \item From the results in Theorems \ref{thm:2rdsa-x} and \ref{thm:2rdsa-H}, we know that $x_n$ and $\nabla^2 f(x_n)$ converge to $x^*$ and $\nabla^2 f(x^*)$, respectively. Thus,
$\Gamma_n \rightarrow a_0$, $\Phi_n \rightarrow -a_0 \nabla^2 f(x^*)^{-1}$. Moreover, $T_n$ is identical to that in 1RDSA and hence, $T_n \rightarrow 0$ if $\gamma > \alpha/6$ and if $\gamma=\alpha/6$, then the limit of $T_n$ is the vector $T$ as defined in Theorem \ref{thm:1rdsa-asymp-normal}. These observations together imply that condition (2.2.1) of \cite{fabian1968asymptotic} is satisfied.   
\item $V_n$ is also identical to that in 1RDSA and hence, $E(V_n V_n\tr\mid x_n) \rightarrow \frac{1}{4} \delta_0^{-1} \sigma^2 I$. This implies condition (2.2.2) of \cite{fabian1968asymptotic} is satisfied.
\item Condition (2.2.3) can be verified  using arguments that are the same as those in \cite{spall} for first-order SPSA.  
\end{itemize}
Now, applying Theorem 2.2 of \cite{fabian1968asymptotic}, it is straightforward to obtain the expressions for the mean $\mu$ and covariance matrix $\Gamma$ of the limiting Gaussian distribution.
\end{proof}

%%%%%%%%%%%%%%%%%%%%%%%%%%%%%%%%%%%%%%%%%%%%%%%%%%%%%%%%%%%%%%%%%%%%%%%%%%%%%%%%%%%%%%%%%%%%%%%%%%%%%%%%%%%%%%%%%%%%%%%%%%%%%%%
%%%%%%%%%%%%%%%%%%%%%%%%%%%%%%%%%%%%%%%%%%%%%%%%%%%%%%%%%%%%%%%%%%%%%%%%%%%%%%%%%%%%%%%%%%%%%%%%%%%%%%%%%%%%%%%%%%%%%%%%%%%%%%%

\section{Additional Numerical Results}
\label{sec:appendix-results}
%%%%% new results
Tables \ref{tab:1rdsa-epsilon} and \ref{tab:2rdsa-epsilon} present the results from a sensitivity study conducted for the asymmetric Bernoulli variants of 1RDSA and 2RDSA, respectively. 

\begin{table}[h]
\caption{NMSE of asymmetric Bernoulli perturbations based RDSA schemes as a function of distribution parameter $\epsilon$, with $\sigma=0.001$ and using $2000$ function measurements:\\ standard error from $1000$ replications shown after $\pm$}
\label{tab:epsilon-results}
\centering
\begin{tabular}{cc}
\begin{subfigure}[b]{0.4\textwidth}
\centering
\scriptsize
    \begin{tabular}{|c|c|}    
    \toprule
    \multirow{2}{*}{$\bm{\epsilon}$ \textbf{value}} & \multirow{2}{*}{\textbf{NMSE}} \\[1.7em]     
    \midrule
$0.000001$  & $3.38 \times 10^{-2}  \pm 4.87 \times 10^{-4}$ \\
$0.00001$   & $3.38 \times 10^{-2}  \pm 4.87 \times 10^{-4}$ \\
$0.0001$    & $3.38 \times 10^{-2}  \pm 4.87 \times 10^{-4}$ \\
$0.001$     & $3.38 \times 10^{-2}  \pm 4.87 \times 10^{-4}$ \\
$0.01$      & $3.39 \times 10^{-2}  \pm 4.87 \times 10^{-4}$ \\
$0.1$       & $3.38 \times 10^{-2}  \pm 4.81 \times 10^{-4}$ \\
$0.2$       & $3.37 \times 10^{-2}  \pm 4.87 \times 10^{-4}$ \\
$0.5$       & $3.42 \times 10^{-2}  \pm 5.00 \times 10^{-4}$ \\
$\bm{1}$    & $3.54 \times 10^{-2}  \pm 5.09 \times 10^{-4}$ \\
$2$         & $3.87 \times 10^{-2}  \pm 5.76 \times 10^{-4}$ \\
$5$         & $5.21 \times 10^{-2}  \pm 8.10 \times 10^{-4}$ \\
    \bottomrule
    \end{tabular}%
\caption{1RDSA-AsymBer}
\label{tab:1rdsa-epsilon}
\end{subfigure}
&
\begin{subfigure}[b]{0.4\textwidth}
\centering
\scriptsize
    \begin{tabular}{|c|c|}
    \toprule
    \multirow{2}{*}{$\bm{\epsilon}$ \textbf{value}} & \multirow{2}{*}{\textbf{NMSE}} \\
     & \\
    \midrule    
    $0.000001$  & $7.21 \times 10^{-2} \pm 8.95 \times 10^{-4}$ \\
    $0.00001$   & $7.93 \times 10^{-2} \pm 1.46 \times 10^{-3}$ \\
    $0.0001$    & $6.24 \times 10^{-2} \pm 1.34 \times 10^{-3}$ \\
    $0.001$     & $8.39 \times 10^{-2} \pm 2.36 \times 10^{-3}$ \\
    $0.01$      & $8.32 \times 10^{-2} \pm 4.04 \times 10^{-2}$ \\
    $0.1$       & $2.35 \times 10^{-1} \pm 1.13 \times 10^{-2}$ \\
    $0.2$       & $1.02 \times 10^{-1} \pm 9.10 \times 10^{-3}$ \\
    $0.5$       & $1.45 \times 10^{-4} \pm 1.43 \times 10^{-4}$ \\
    $1$         & $2.24 \times  10^{-6} \pm 3.67 \times 10^{-8}$ \\
    $2$         & $2.24 \times 10^{-6} \pm 3.35 \times 10^{-8}$ \\
    $5$         & $2.85 \times 10^{-6} \pm 4.60 \times 10^{-8}$ \\
      \bottomrule
    \end{tabular}%
\caption{2RDSA-AsymBer}
\label{tab:2rdsa-epsilon}
\end{subfigure}
\end{tabular}
\end{table}